\documentclass{tran-l}

\usepackage{graphicx}
\usepackage{amsfonts,amsmath,latexsym,amssymb,amsthm}

\newtheorem{theorem}{Theorem}[section]
\newtheorem{lemma}[theorem]{Lemma}
\newtheorem{corollary}[theorem]{Corollary}
\newtheorem{proposition}[theorem]{Proposition}

\theoremstyle{definition}
\newtheorem{definition}[theorem]{Definition}

\theoremstyle{remark}
\newtheorem{remark}[theorem]{Remark}

\numberwithin{equation}{section}

\newcounter{eqncounter}

\numberwithin{equation}{eqncounter}

\def\pw{\mathfrak{p}}

\def\H2{H_2}

\def\ALS{ALS}
\def\m{m}
\def\wK{w_K}
\def\M{M}
\def\ggg{g_{\max}}

\def\IR{\mathbb R}
\def\IC{\mathbb C}
\def\IZ{\mathbb Z}

\def\IP{\mathbb P}
\def\IQ{\mathbb Q}

\def\IU{\mathbb U}

\def\P{\mathcal{P}}

\def\IL{\mathcal{L}}

\def\OD{\Omega}
\def\abbmin{\nu}
\def\en{\mathcal{N}}
\def\hen{H_\mathcal{N}}
\def\balf{\vec{\alpha}}
\def\bom{{\bf \omega}}
\def\bz{{\bf z}}

\def\bz0{{\bf z^{(0)}}}
\def\bz1{{\bf z^{(1)}}}

\def\grmnm1{\mathcal{M}^*_K(n-1,X)}

\def\grmnmnew1{\mathcal{M}^*_K(n-1,X^n)}

\renewcommand{\vec}[1]{\mbox{\boldmath$#1$}}
\def\ord{\mathop{\rm ord}\nolimits}
\newcommand{\tr}{\operatorname{tr}}                  

\def\Oseen{{\mathcal{O}}}
\def\A{{\mathfrak{A}}}
\def\C{{\mathfrak{C}}}

\def\D{{\mathfrak{D}}}
\def\B{{\mathfrak{B}}}

\def\L{{\mathfrak{L}}}
\def\T{{\mathfrak{T}}}
\let\rho\varrho
\def\ta_0{\tau}

\def\ti{{\tau_{{\bf i}}}}
\def\i{{{\bf i}}}
\def\g0{\ta_0\sigma(\C_0^{-1}\D)}
\def\G0{\tau\Lambda(\D)}
\def\Lam{\Lambda}
\def\Lamen{\Lambda_{\en}}

\def\Lamst{\Lambda^*}

\def\Qbar{\overline{\IQ}}

\def\kbar{\overline{k}}

\def\tr{tr_{{\bf i}}}
\def\etr{etr_{{\bf i}}}
\def\vx{{\bf x}}
\def\vy{{\bf y}}
\def\vz{{\bf z}}

\def\vdelta{\vec{\delta}}

\def\v0{\vec{0}}
\def\d_1{\kappa}
\def\Vol{\textup{Vol }}

\def\Cl{Cl}

\def\Da{D}
\def\ka{c}

\begin{document}

\title{Counting primitive points of bounded height}


\author{Martin Widmer}
\address{Mathematisches Institut \\
  Universit\"at Basel \\
  Rheinsprung 21 \\
  4051 Basel \\
  Switzerland}
\curraddr{Department of Mathematics \\
  University of Texas at Austin \\
  1 University Station C1200 \\
  Austin, Texas 78712 \\
  U.S.A}
\email{widmer@math.utexas.edu}
\thanks{The author was supported by NSF Grant \#118647}


\date{November 25, 2008}

\dedicatory{}

\begin{abstract}
Let $k$ be a number field and $K$ a finite extension
of $k$.
We count points of bounded height
in projective space over the field $K$ generating the extension 
$K/k$. As the height gets large we derive asymptotic estimates
with a particularly good error term respecting the extension $K/k$.
In a future paper we will use these results to get asymptotic estimates for the 
number of points of fixed degree over $k$. 
We also introduce the notion
of an adelic Lipschitz height generalizing that of Masser and Vaaler. This will lead to further applications involving
points of fixed degree on linear varieties and algebraic numbers of fixed degree satisfying certain
subfield conditions.
\end{abstract}

\maketitle

\tableofcontents

\section{Introduction}
Let $K$ be a number field of degree $d$ and write $\IP^n(K)$ for the projective space of dimension $n$
over $K$. Denote by $H$ the non-logarithmic absolute
Weil height on $\IP^n(K)$; the definition is given in Section \ref{secALH}. 
A well-known result
due to Northcott (\cite{26} Theorem) implies that
$Z_H(\IP^n(K),X)$,
the number of points in $\IP^n(K)$ with height not larger than $X$,
is finite for each positive real number $X$.
Schanuel \cite{25} had proved the following asymptotic estimate. As $X$ tends to infinity one has
\begin{alignat}3
\label{Scha1}
Z_H(\IP^n(K),X)=S_K(n)X^{d(n+1)}+O(X^{d(n+1)-1}\log X).
\end{alignat}
The logarithm can be omitted in all cases except for
$n=d=1$ and the constant implicit in $O$ depends on $K$ and $n$ only.
The constant $S_K(n)$ in the main term depends on the detailed field structure and involves all classical field 
invariants.\\

More recently Masser and Vaaler \cite{1} introduced heights where the maximum norms
at the infinite places are replaced by more general so called Lipschitz distance functions,
let us call them Lipschitz heights.
Masser and Vaaler generalized Schanuel's result to Lipschitz heights and simplified the original proof
considerably. Their main application of this generalization is an asymptotic
counting result on algebraic numbers of bounded height and fixed degree. But they also deduce
other counting results e.g. on algebraic subgroups of the multiplicative group $\mathbb{G}_m^{n+1}$ 
with bounded degree.\\

In the present paper we generalize these results in several respects. First we allow also arbitrary norms at a finite number of finite places in the spirit of an adelic viewpoint. Secondly we make the constant in the error term more explicit in the sense of Schmidt \cite{14} and Gao \cite{7}. Thirdly, also in this sense, we show that this constant goes rapidly to zero as the field $K$ becomes more complicated, under the necessary condition that the  counting is restricted to primitive points. Fourthly we generalize the primitivity condition to involve an arbitrary subfield $k$ of $K$. Fifthly we express the constant in terms of some new invariant $\delta(K/k)$ which itself generalizes a quantity $\delta(K/Q)$ introduced by Roy and Thunder \cite{8}. 
Sixthly we present an improvement in terms of certain refined
quantities $\delta_g(K/k)$. And finally, more on the technical level, we calculate the dependence
on the Lipschitz functions themselves.\\

We carry out these various generalizations not only for their own sake, but also with definite applications in mind, which we intend to publish in future papers.
Here is a more detailed discussion.
First of all, the adelic generalization is natural in view of the 
equal status of all places on a number field.
But it is also essential so that we can deduce some new results about counting points on subspaces. 
Let us illustrate this with a simple example.
The height of a point on the plane defined by the equation $2x+3y-z=0$ involves expressions
\begin{alignat}3
\label{exintro}
\max\{|x|_v,|y|_v,|z|_v\}=\max\{|x|_v,|y|_v,|2x+3y|_v\} 
\end{alignat}
with valuations $v$ corresponding to various places. If the place is infinite, then the right-hand side of (\ref{exintro}) is a function of $x,y$ as allowed in \cite{1}; and if the place is finite, then it is simply $\max\{|x|_v,|y|_v\}$ as required in \cite{1}. But if we change the equation to $2x+3y-5z=0$ then the left hand-side of (\ref{exintro}) is $\max\{|x|_v,|y|_v,|(2x+3y)/5|_v\}$ which is not $\max\{|x|_v,|y|_v\}$ at places over the prime 5. Hence we must be prepared
to allow modifications on the max-norm not only at the
infinite places but also at a finite number of finite
places.\\

In \cite{art3} 
we will prove a counting result for points of fixed degree on a linear 
projective variety. This generalizes a result of Thunder (Theorem 1 in \cite{34}).
Thunder \cite{78} introduced twisted heights
where all places are considered in a perfectly equal manner.
But twisted heights are more restrictive at the infinite
places and are therefore not applicable to deduce the results
in \cite{art4}, mentioned in the last paragraph of this section.\\

Regarding the second and third generalizations mentioned above, Schmidt \cite{14} in 1995 considered for quadratic $K$ the set $\IP^n(K/\IQ)$ of primitive points of $\IP^n(K)$ whose affine coordinates generate (over $\IQ$) the whole field $K$. The main term in (\ref{Scha1}) is not changed, but he could replace the error term (for $d=2$) by
\begin{alignat}3
\label{Schmerror1}
O\left(\frac{\sqrt{h_KR_K\log(3+h_KR_K)}}{|\Delta_K|^{n/2}}X^{2n+1}\right)
\end{alignat}
where $h_K$ is the class number, $R_K$ denotes the regulator, $\Delta_K$
is the discriminant and the constant in $O$ depends only on $n$ but is independent
of the field $K$.
It is not difficult to see that such a good estimate cannot hold without the primitivity condition. Schmidt's purpose was to deduce asymptotic results for counting points of $\IP^n$ quadratic over $\IQ$. This he did by the simple but bold idea of summing over all quadratic fields $K$, when the large power of the discriminant in (\ref{Schmerror1}) is necessary for convergence. Everything was generalized to arbitrary $K$ by Gao \cite{7}, also in 1995. He extended (\ref{Schmerror1}) and also obtained a more complicated version with better summatory properties. This enabled him to deduce asymptotic results for counting points of $\IP^n$ of fixed degree $e$ over $\IQ$ provided $n>e$. However, Gao's work remains unpublished.\\

Regarding the fourth and fifth generalizations, our motivation is to extend Gao's results to count points of $\IP^n$ of fixed degree $e$ over a fixed number field $k$. This problem was already considered by Schmidt in \cite{22}. In the present paper we express our error terms like (\ref{Schmerror1}) using the quantities $\delta(K/k)$, which also have better summatory properties than the discriminant. Those for the discriminant are still governed by difficult conjectures such as Linnik's Conjecture (see \cite{56}). The latter is proved only for very special cases although great progress was achieved by the recent work of Ellenberg and Venkatesh \cite{56}. Anyway, by using $\delta$ we are able to deduce asymptotic results for counting points of $\IP^n$ of fixed degree $e$ over $k$ provided $n>4e$. 
And it is the refined quantities $\delta_g(K/k)$ that enable us to improve this to $n$ about $5e/2$.\\

Finally the Lipschitz functions in the heights are characterised by certain
parametrizations involving Lipschitz constants, and we develop a formalism for
calculating with these.\\

Let us informally present a special case of our main result Theorem \ref{prop3}. We are now counting the set $\IP^n(K/k)$ of primitive points of $\IP^n(K)$ whose affine coordinates generate over $k$ the whole field $K$; but this time with respect to an adelic Lipschitz height $\en$. We then generalize and improve (\ref{Scha1}) in the style of (\ref{Schmerror1}) to
\begin{alignat}1 
\label{introresultat}
Z_{\en}(\IP^n(K/k),X)=S_{\en}(n)X^{d(n+1)}+O\left(A_{\en}\frac{h_KR_K}{\delta(K/k)^{d(n+1)/2-1}}X^{d(n+1)-1}{\IL}_{\en}\right),
\end{alignat}
now with the constant implied in the $O$ depending only on $d$ and $n$. Here $S_{\en}(n)$ is related to certain volumes of unit balls and lattice determinants, and $A_{\en}$ is related to the Lipschitz constants for unit spheres and the norms; while ${\IL}_{\en}$ is logarithmic in $X$.\\

Our Theorem \ref{prop3} sharpens (\ref{introresultat}) yet further in terms of the $\delta_g(K/k)$. It
has various applications such as counting points of fixed degree in $\IP^n(\kbar)$ 
($\kbar$ denotes an algebraic closure of $k$)
and on linear subvarieties of $\IP^n(\kbar)$ defined over $k$
(see \cite{art3}). Due to the $n>5e/2$ condition we need the dimension of the underlying variety to be sufficiently large when compared with the degree. In particular we are unable to count quadratic points on a line.
But Theorem \ref{prop3} leads also to a generalized version of Proposition in \cite{1}
(in fact with a particularly good error term)
and it is most likely that using this generalized
proposition and following
the ideas of Masser and Vaaler in \cite{1}
one can in fact deduce the
asymptotics for points of fixed degree on an arbitrary line, despite the dimension being so small.\\

Let us mention briefly some other applications of Theorem \ref{prop3}.
Thanks to \cite{art3} we can sometimes sum over linear subvarieties rather than number fields. In this way we can obtain the asymptotics for points over a fixed number field on a non-linear hypersurface like that defined by $x-yz^r=0$. Here the main term involves the so-called height zeta function. 
Or more ambitiously we can occasionally sum over both linear subvarieties and number fields to get the asymptotics for points of fixed degree on more elaborate non-linear varieties like that defined by
\begin{alignat*}3
x_{1}-y_{1}z^r=\dots=x_{n}-y_{n}z^r=0.
\end{alignat*}
Finally let us mention that Theorem \ref{prop3} can be used to derive
a refinement of Masser and Vaaler's result (Theorem in \cite{37})
on counting algebraic numbers. Let $m$ and
$n$ be natural numbers. Instead of counting all algebraic numbers $\alpha$
of degree $mn$ as in \cite{37} we consider only those numbers $\alpha$ such that $\IQ(\alpha)$
contains a subfield of degree $m$.
If $n$ is much larger than $m$ Theorem \ref{prop3} can be applied to 
get the correct asymptotics. For instance the asymptotics for points of degree
$32$ involve $X^{1056}$ while the number of points of degree $32$ generating a field with a 
quadratic subfield has only order of magnitude $X^{544}$.
This leads
also to information on the distribution of number fields of 
degree $d$ containing a proper intermediate field
if ordered via the function $\delta$; for more details we refer to
\cite{art4}.\\

We close the introduction with a few remarks about the structure of our paper.\\

In Section \ref{secALH} we introduce the notion of an adelic Lipschitz system
leading to an adelic Lipschitz height on $\IP^n(K)$.
The main result Theorem \ref{prop3} is stated in Section \ref{chap3sec2}.
Furthermore we show that it implies (\ref{introresultat}) as our Corollary \ref{prop2}.
The problem of estimating $Z_{\en}(\IP^n(K/k),X)$ is reduced to counting lattice
points in a certain bounded region $S$ of $\IR^{D}$. In Section \ref{1subsec1} we recall some basic
facts about lattices in general. In Section \ref{1subsec3} we develop the basic counting technique for lattice
points which relies on parameterization maps of the boundary
$\partial S$ satisfying a Lipschitz condition. 
In Section \ref{thebasicset} we introduce the set $S=S_F(T)$ where the counting will be carried out.
Then in Section \ref{lpb} we show that this set satisfies the necessary Lipschitz conditions; but in order not to
distract the reader too much from the basic line of the proof we postpone the somewhat tedious and lengthy proof to the appendix.
However, it turns
out that we are faced with a serious problem when applying the counting method since the Lipschitz constants for our boundary
$\partial S$ are far too large, resulting in a very bad error term.
In \cite{14} (which deals with $d=2$) Schmidt shows a way out of this misery by splitting up the set $S$ in several subsets and applying a suitable
linear transformation on each of them. Section \ref{schmidtstrick} is dedicated to the extension of Schmidt's approach from $d=2$ to arbitrary $d$.  As in Gao's work \cite{7} this extension is relatively straightforward.
The primitivity condition of $\IP^n(K/k)$ translates directly into an arithmetic property for the lattice points. 
In Section \ref{Estminima} we translate this into a geometric property saying that the
length of each lattice point which gives a contribution to $\IP^n(K/k)$ is bounded below nicely in terms of $\delta_g(K/k)$.
In Section \ref{subseccounting} we apply the counting techniques of Section \ref{1subsec3} to obtain estimates for the number of lattice points in $S_F(T)$ using the geometric property established in Section \ref{Estminima}. In this way $\delta_g(K/k)$ enters the error estimates. Finally in Section \ref{2endofproof} we are in position to prove Theorem \ref{prop3}.

\section*{Acknowledgements}
I am indebted to my Ph.D. adviser David Masser for substantial improvements on an earlier version
of this article, plenty of delightful discussions and finally for motivating me to 
work on the problems considered here.
I also would like to thank Jeffrey Thunder and Jeffrey Vaaler for fruitful conversations and comments.
Finally, I thank the referee for carefully reading this manuscript and for valuable suggestions.
This work was financially supported by the Swiss National Science Foundation.

\section{Definitions}\label{secALH}
In 1967 Schmidt \cite{86} introduced heights where the max-norm at the infinite places (see (\ref{height}) below)
is replaced by a fixed but arbitrary distance function.
Masser and Vaaler's Lipschitz heights
in \cite{1} are more flexible since they allow different Lipschitz distance functions at the infinite places.
Adelic Lipschitz heights are a natural generalization of  Masser and Vaaler's Lipschitz heights.
Before we can define adelic Lipschitz heights we have to fix some basic notation.
For a detailed account on heights we refer the reader to \cite{BG} and \cite{3}.\\

Let $K$ be a finite extension
of $\IQ$ of degree $[K:\IQ]=d$. By a place $v$ of $K$ we mean an equivalence class of non-trivial absolute values on $K$.
The set of all places of $K$ will be denoted by $M_K$.
For each $v$ in $M_K$ we write $K_v$ for the completion
of $K$ with respect to the place $v$ and $d_v$
for the local degree defined by
$d_v=[K_v:\IQ_v]$ 
where $\IQ_v$ is a completion with respect to the place which extends to $v$. A place $v$ in $M_K$ corresponds either to a non-zero prime ideal $\pw_v$
in the ring of integers $\Oseen_K$ or to a complex embedding
$\sigma$ of $K$ into $\IC$.
If $v$ comes from a prime ideal we call $v$
a finite or non-archimedean place indicated by $v\nmid \infty$ and if $v$ corresponds to
an embedding we say $v$ is an infinite or archimedean
place abbreviated to $v\mid \infty$.
For each place in $M_K$ we choose a representative
$|\cdot|_v$,
normalized in the following way:
if $v$ is finite and $\alpha\neq 0$ we set by convention
\begin{alignat*}3
|\alpha|_{v}=N\pw_v^{-\frac{\ord_{\pw_v}(\alpha\Oseen_K)}{d_v}}
\end{alignat*}
where $N\pw_v$ denotes the norm of $\pw_v$
from $K$ to $\IQ$ and $\ord_{\pw_v}(\alpha\Oseen_K)$
is the power of $\pw_v$ in the prime ideal decomposition
of the fractional ideal $\alpha\Oseen_K$.
Moreover we set
\begin{alignat*}3
|0|_{v}=0.
\end{alignat*}
For $v$ infinite we define
\begin{alignat*}3
|\alpha|_{v}=|\sigma(\alpha)|
\end{alignat*}
where $|\cdot|$ is the usual complex modulus.
Suppose $\alpha$ is in $K^*=K\backslash\{0\}$ then 
$|\alpha|_v\neq 1$
holds only for a finite number of places $v$.\\

Throughout this article $n$ will denote a natural number, which 
means a positive rational integer. The height
on $K^{n+1}$ is defined by
\begin{alignat}3
\label{height}
H(\alpha_0,...,\alpha_n)=\prod_{M_K}\max\{|\alpha_0|_v,...,|\alpha_n|_v\}^{\frac{d_v}{d}}.
\end{alignat}
Due to the remark above this is in fact a finite product. 
Furthermore this definition is independent of the field $K$ containing
the coordinates (see \cite{BG} Lemma 1.5.2 or \cite{3} pp.51-52) and therefore
defines a height on $\Qbar^{n+1}$ for an algebraic closure
$\Qbar$ of $\IQ$.
The well-known \it product formula \em (see \cite{BG} Proposition 1.4.4) asserts that
\begin{alignat*}3
\prod_{M_K}|\alpha|_v^{d_v}=1 \text{ for each $\alpha$ in $K^*$}.
\end{alignat*}
This implies in particular that the value of the height in (\ref{height}) does not change if we multiply each coordinate with a fixed element of $K^*$. 
Therefore one can define a height on points
$P=(\alpha_0:...:\alpha_n)$ in $\IP^n(\Qbar)$ by
\begin{alignat}3
\label{heightproj}
H(P)=H(\alpha_0,...,\alpha_n)
\end{alignat}
and moreover $H(\balf)\geq 1$ for $\balf\in \Qbar^{n+1}\backslash\{\v0\}$.
The equations (\ref{height}) and (\ref{heightproj}) define the absolute non-logarithmic
projective Weil height or simpler Weil height.\\

Let $r$ be the number of  
real embeddings and $s$ the number of pairs of complex conjugate embeddings
of $K$ so that $d=r+2s$.
For every place $v$ we fix a
completion $K_v$ of $K$ at $v$. There is a 
value set 
\begin{alignat*}3
\Gamma_v=\{|\alpha|_v;\alpha \in K_v\}.
\end{alignat*}
It is $[0,\infty)$ for $v$ archimedean
and 
\begin{alignat*}3
\{0,(N\pw_v)^{0},(N\pw_v)^{\pm 1/d_v},(N\pw_v)^{\pm 2/d_v},...\}
\end{alignat*}
otherwise.
For $v \mid \infty$ we identify $K_v$ with $\IR$ or
$\IC$ respectively and we identify $\IC$ with
$\IR^2$ via $\xi\longrightarrow (\Re(\xi),\Im(\xi))$
where we used $\Re$ for the real and $\Im$ for the
imaginary part of a complex number.\\

For a vector $\vx$ in $\IR^n$ we write $|\vx|$ 
for the euclidean length of $\vx$.
$\Da$ and $\M$ will always stand for a natural
number while $L$ will denote a non-negative real number.
\begin{definition}
Let $S$ be a subset of $\IR^\Da$ and let $\ka$ be an integer with $0\leq \ka\leq \Da$. We say $S$ is in Lip$(\Da,\ka,\M,L)$ 
if there are $\M$ maps 
$\phi:[0,1]^{\Da-\ka}\longrightarrow \IR^\Da$
satisfying a Lipschitz condition
\begin{alignat}3
\label{lipcond1}
|\phi(\vx)-\phi(\vy)|\leq L|\vx-\vy|
\end{alignat}
such that $S$ is covered by the images
of the maps $\phi$. For $\ka=\Da$ this is to be 
interpreted simply as the finiteness of
the set $S$.
\end{definition}
We call $L$ a Lipschitz constant for $\phi$.
For $\ka=\Da$ we interpret $[0,1]^{\Da-\ka}$ as $\{0\}\subseteq\IR$ and then
$\M>0$ is simply an upper bound for the cardinality of $S$
and any non-negative $L$ is allowed. By definition the empty set
lies in Lip$(\Da,\ka,\M,L)$ for any natural numbers $\Da$,$\M$ any $\ka$ in 
$\{0,1,2,...,\Da\}$ and any 
non-negative $L$. However, in our applications $\ka$ will be 
$1$ or $2$.
\begin{definition}[Adelic Lipschitz system]\label{defALS}
An adelic Lipschitz system ($\ALS$) 
$\en_K$ or simply $\en$ on $K$ (of dimension $n$) is
a set of continuous maps
\begin{alignat}3
\label{Abb1}
N_v: K_v^{n+1}\rightarrow \Gamma_v \quad v \in M_K
\end{alignat}
such that
\begin{alignat*}3
(i)&\text{ } N_v({\vz })=0 \text{ if and only if } {\vz} ={\v 0},\\
(ii)&\text{ } N_v(\omega {\vz})=|\omega|_v N_v({\vz}) \text{ for all
$\omega$ in $K_v$ and all ${\vz}$ in $K_v^{n+1}$},\\
(iii)& \text{ if } v \mid \infty:
\{{\vz}: N_v({\vz})=1\} \text{
is in Lip$(d_v(n+1),1,\M_v,L_v)$ for some $\M_v, L_v$},\\
(iv)& \text{ if } v \nmid \infty: N_v({\vz_1}+{\vz_2})
\leq \max\{N_v({\vz}_1),N_v({\vz}_2)\} \text{ for all ${\vz}_1,{\vz}_2$ in $K_v^{n+1}$}.
\end{alignat*}
Moreover we assume that only a finite number
of the functions $N_v(\cdot)$ are different
from 
\begin{alignat}3
\label{Nvmaxnorm}
N_v(\vz)=\max\{|z_0|_v,...,|z_n|_v\}.
\end{alignat}
\end{definition}
If we consider only the functions $N_v$ for $v\mid\infty$
then we get an $(r,s)$-Lipschitz system (of dimension $n$)
in the sense of Masser and Vaaler \cite{1}.
With $\M_v$ and $L_v$ from $(iii)$ we define 
\begin{alignat*}1
\M_{\en}&=\max_{v \mid \infty}M_v,\\
L_{\en}&=\max_{v \mid \infty}L_v.
\end{alignat*}
We say that $\en$
is an $\ALS$ with associated constants $\M_{\en}, L_{\en}$.
For $v\mid\infty$ we call $N_v$ a \it Lipschitz distance function \rm  
(of dimension $n$). 
The set defined in $(iii)$ is the boundary 
of the set ${\bf B}_v=\{{\vz};N_v({\vz})<1\}$
and therefore ${\bf B}_v$ is a bounded symmetric open star-body
in $\IR^{n+1}$ or $\IC^{n+1}$ (see also \cite{1} p.431). In particular ${\bf B}_v$ has a finite volume $V_v$.\\

Let us consider the system
where $N_v$ is as in (\ref{Nvmaxnorm}) for all places $v$.
If $v$ is an infinite place then  
${\bf B}_v$ is a
cube for $d_v=1$ and the complex analogue
if $d_v=2$. Their boundaries are clearly
in Lip$(d_v(n+1),1,M_v,L_v)$ most naturally
with $M_v=2n+2$ maps and $L_v=2$
if $d_v=1$ and 
with $M_v=n+1$ maps and for example $L_v=2\pi\sqrt{2n+1}$
if $d_v=2$.
This system is the standard example for an 
adelic Lipschitz system.\\

We claim that for any $v\in M_K$ there is a $c_v$ in the value group
$\Gamma_v^*=\Gamma_v\backslash\{0\}$ with
\begin{alignat}3
\label{Nineq1}
N_v({\vz})\geq c_v\max\{|z_0|_v,...,|z_n|_v\}
\end{alignat}
for all $\vz=(z_0,...,z_n)$ in $K_v^{n+1}$.
For if  $v$ is archimedean then ${\bf B}_v$ is
bounded open and contains the origin.
Since $\Gamma_v^*$ contains arbitrary small
positive numbers the
claim follows by $(ii)$.
Now for $v$ non-archimedean $N_v$ and $\max\{|z_0|_v,...,|z_n|_v\}$ define norms on 
the vector space $K_v^{n+1}$ over the complete field $K_v$.
But on a finite dimensional vector space over a complete
field all norms are equivalent (\cite{2} Corollary 5. p.93)
hence (\ref{Nineq1}) remains true for a suitable choice
of $c_v$.\\

So let $\en$ be an $\ALS$ on $K$ of dimension $n$. For
every $v$ in $M_K$ let
$c_v$ be an element of $\Gamma_v^*$,
such that $c_v\leq 1$ and (\ref{Nineq1}) holds.
Due to (\ref{Nvmaxnorm}) we can assume that $c_v\neq 1$ only for a finite number of places $v$.
Define
\begin{alignat}3
\label{defcfin}
C^{fin}_{\en}&=\prod_{v}c_v^{-\frac{d_v}{d}}\geq 1
\end{alignat}
where the product runs over all finite $v$.
Next for the infinite part we define
\begin{alignat}3
\label{defcinf}
C^{inf}_{\en}&=\max_{v}\{c_v^{-1}\}\geq 1
\end{alignat}
where now $v$ runs over all infinite $v$.\\

Multiplying the finite and the infinite part 
gives rise to another constant 
\begin{alignat}3
\label{defc}
C_{\en}&=C^{fin}_{\en}C^{inf}_{\en}.
\end{alignat}
It will turn out that besides $\M_{\en}$ and $L_{\en}$ this is another important quantity for an $\ALS$. So we say that \it $\en$
is an $\ALS$ with associated constants $C_{\en},\M_{\en},L_{\en}$.\rm 
\begin{remark}\label{normconvex}
Let $v$ be an infinite place.
Suppose $N_v:K_v^{n+1}\longrightarrow [0,\infty)$
defines a norm, so that 
$N_v({\vz_1}+{\vz_2})
\leq N_v({\vz}_1)+N_v({\vz}_2)$. Then
${\bf B}_v$ is convex and (\ref{Nineq1}) combined with 
(\ref{defcfin}), (\ref{defcinf}) and (\ref{defc}) shows that ${\bf B}_v$ lies in
$B_0(C_{\en}\sqrt{n+1})$.
This implies (see Theorem A.1 in \cite{WiThesis}) that  $\partial{\bf B}_v$ lies in Lip$(d_v(n+1),1,1,8{d_v}^2(n+1)^{5/2}C_{\en})$.
\end{remark}

We denote by $\sigma_1,...,\sigma_d$ 
the embeddings from $K$ to $\IR$ or
$\IC$ respectively, ordered such that
$\sigma_{r+s+i}=\overline{\sigma}_{r+i}$ for 
$1\leq i \leq s$.
We write
\begin{alignat}3
\label{sigd}
&\sigma:K\longrightarrow \IR^r\times\IC^s\\
\nonumber&\sigma(\alpha)=(\sigma_1(\alpha),...,\sigma_{r+s}(\alpha)).
\end{alignat}
Sometimes it will be more readable to omit the brackets
and simply to write $\sigma\alpha$.
We identify $\IC$ in the usual way with $\IR^2$
and extend $\sigma$ componentwise to get a map
\begin{alignat}3
\label{sigD}
\sigma:K^{n+1}\longrightarrow \IR^{D}
\end{alignat}
where $D=d(n+1)$.
On $\IR^{D}$ we use $|\cdot|$ for the
usual euclidean norm.
Let $\sigma_v$ be the canonical embedding of $K$ in $K_v$
again extended componentwise on $K^{n+1}$. 
\begin{definition}
Let $\D\neq 0$ be a fractional ideal in $K$ and let $\en$ be
an $\ALS$ of dimension $n$. We define
\begin{alignat}3
\label{defLamen}
\Lamen(\D)=\{\sigma(\balf); \balf \in K^{n+1}, 
N_v(\sigma_v\balf)\leq |\D|_v \text{ for all finite }v \}
\end{alignat}
where $|\D|_v=N\pw_v^{-\frac{\ord_{\pw_v}\D}{d_v}}$.
\end{definition}
It is easy to see that 
$\Lamen(\D)$ is an additive subgroup of $\IR^D$.
Now assume $B\geq 1$ and 
$|\sigma(\balf)|\leq B$; then (\ref{Nineq1}) implies
$H(\balf)^d\leq (BC_{\en}^{fin})^d N\D^{-1}$ and by Northcott's Theorem
we deduce that $\Lamen(\D)$ is discrete.
The same argument as for (\ref{Nineq1}) yields positive real numbers $C_v$, one for each non-archimedean place $v\in M_K$,
with $N_v({\vz})\leq C_v\max\{|z_0|_v,...,|z_n|_v\}$ for all $\vz=(z_0,...,z_n)$ in $K_v^{n+1}$
and $C_v=1$ for all but finitely many non-archimedean $v\in M_K$.
Thus there exists an ideal $\C_1\neq 0$ in $\Oseen_K$ with 
$|\C_1|_v\leq 1/C_v$ for all non-archimedean places $v\in M_K$.
This means that $\sigma(\C_1\D)^{n+1}\subseteq \Lamen(\D)$.
It is well-known that the additive group $\sigma(\C_1\D)^{n+1}$ has maximal rank in $\IR^D$.
Therefore $\Lamen(\D)$ is  a discrete additive subgroup of $\IR^D$
of maximal rank. Hence $\Lamen(\D)$
is a lattice. Notice that for $\varepsilon$ in $K^*$ one has
\begin{alignat}3
\label{deltawelldef1}
\det\Lamen((\varepsilon)\D)=
|N_{K/\IQ}(\varepsilon)|^{n+1}\det\Lamen(\D).
\end{alignat}
Therefore
\begin{alignat}3
\label{idkl}
\Delta_{\en}(\mathcal{D})=\frac{\det\Lambda_{\en}(\D)}
{N\D^{n+1}}
\end{alignat}
is independent of the choice of the representative $\D$
but depends only on the  ideal class $\mathcal{D}$ of $\D$. 
Let $\Cl$ be the set of ideal classes.
We define
\begin{alignat}3
\label{defVfin1}
V_{\en}^{fin}=2^{-s(n+1)}|\Delta_K|^{\frac{n+1}{2}}h_K^{-1}
\sum_{\mathcal{D} \in \Cl}\Delta_{\en}(\mathcal{D})^{-1}
\end{alignat}
for the finite part.
The infinite part is defined by 
\begin{alignat*}3
V_{\en}^{inf}=\prod_{v \mid \infty}V_{v}.
\end{alignat*}
By virtue of (\ref{Nineq1}) we observe that
\begin{alignat}3
\label{Vinfbou1}
V_{\en}^{inf}=\prod_{v|\infty} V_v\leq 
\prod_{v|\infty}(2 C^{inf}_{\en})^{d_v(n+1)}=
(2 C^{inf}_{\en})^{d(n+1)}.
\end{alignat} 
We multiply the finite and the infinite part
to get a global volume 
\begin{alignat}3
\label{defVen}
V_{\en}=V_{\en}^{inf}V_{\en}^{fin}.
\end{alignat}
We proceed as in Masser and Vaaler's article 
to obtain a  height.
Let $\en$ be an $\ALS$ on $K$
of dimension $n$. Then the height $H_{\en}$ 
on $K^{n+1}$ is defined by
\begin{alignat*}3
\hen(\balf)=\prod_v N_v(\sigma_v(\balf))^{\frac{d_v}{d}}
\end{alignat*}
where the product is taken over all $v \in M_K$.
The product over the archimedean absolute values will be 
denoted by $\hen^{inf}(\cdot)$ and the one over the non-archimedean absolute values by $\hen^{fin}(\cdot)$.
The product formula
together with $(ii)$ implies that $\hen$ is well-defined on $\IP^n(K)$.
\begin{remark}
Multiplying (\ref{Nineq1}) over all places with 
suitable multiplicities yields 
\begin{alignat}3
\label{HAquiv}
\hen(\balf)\geq C_{\en}^{-1} H(\balf).
\end{alignat}
Thanks to Northcott's Theorem it follows that
$\{P\in \IP^n(K); \hen(P)\leq X\}$ is a finite set for
each $X$ in $[0,\infty)$.
\end{remark}
Let $k$ be a number field and let $K$ be a finite extension of $k$.
For a point $P=(\alpha_0:...:\alpha_n)$ in $\IP^n(K)$ let
$k(P)=k(...,\alpha_i/\alpha_j,...)$ $(0\leq i,j\leq n; \alpha_j\neq 0)$.
We write $\IP^n(K/k)$ for the set of primitive points
\begin{alignat*}3
\IP^n(K/k)=\{P\in \IP^n(K);k(P)=K\}
\end{alignat*}
and 
\begin{alignat*}3
Z_{\en}(\IP^n(K/k),X)=|\{P\in \IP^n(K/k);\hen(P)\leq X\}|
\end{alignat*}
for its counting function with respect to the adelic Lipschitz height $\hen$.\\

Before stating the main result we have to introduce some more basic notation.\\

First of all we need the Schanuel constant from (\ref{Scha1}) 
\begin{alignat}3
\label{Schanuelkonst}
S_K(n)=\frac{h_KR_K}{\wK\zeta_K(n+1)}
\left(\frac{2^{r_K}(2\pi)^{s_K}}{\sqrt{|\Delta_K|}}\right)^{n+1}
(n+1)^{r_K+s_K-1}.
\end{alignat}
Here $h_K$ is the class number, $R_K$ the regulator,
$\wK$ the number of roots of unity in $K$, $\zeta_K$
the Dedekind zeta-function of $K$, $\Delta_K$ the discriminant,
$r_K$ is the number of real embeddings of $K$ and $s_K$ is
the number of pairs of distinct complex conjugate embeddings of $K$.\\

Moreover we need a set $G(K/k)$ and a new invariant $\delta_g(K/k)$.
First for fields $k,K$ with $k\subseteq K$ and $[K:k]=e$ we define
\begin{alignat*}3
G(K/k)=
\{[K_0:k]; \text{$K_0$ is a field with $k\subseteq K_0\subsetneq K$}\}
\end{alignat*}
if $k\neq K$, and we define  
\begin{alignat*}3
G(K/k)=\{1\}
\end{alignat*}
if $k=K$. Clearly $|G(K/k)|\leq e$.
Then for an integer $g\in G(K/k)$ we define
\begin{alignat}3
\label{delta2}
\delta_g(K/k)=
\underset{\alpha,\beta}\inf\{H(1,\alpha,\beta); k(\alpha,\beta)=K,
[k(\alpha):k]=g\}\geq 1
\end{alignat}
and 
\begin{alignat}3
\label{mu2}
\mu_g=m(e-g)(n+1)-1.
\end{alignat}
It will be convenient to use Landau's $O$-notation.
For non-negative real functions $f(X), g(X), h(X)$ we say that
$f(X)=g(X)+O(h(X))$ as $X>X_0$ tends to infinity if there is a constant $C_0$ such that
$|f(X)-g(X)|\leq C_0h(X)$ for each $X>X_0$.
In Section \ref{subseccounting} we will use Vinogradov's  $\ll$ notation. An expression
$A\ll B$ or equivalently $B\gg A$ means that there is a positive constant $c$ depending only on 
$n$ and $d$ such that $A\leq cB$.

\section{The main result}\label{chap3sec2}
The following theorem is the main result of this article. It gives an asymptotic estimate of the counting function $Z_{\en}(\IP^n(K/k),X)$
with a particularly good error term.
\begin{theorem}\label{prop3}
Let $k,K$ be number fields with $k\subseteq K$ and $[K:k]=e$, $[k:\IQ]=m$, $[K:\IQ]=d$. Let $\en$ be an adelic Lipschitz system of dimension $n$ 
on $K$ with associated constants $C_{\en},L_{\en},\M_{\en}$.
Write
\begin{alignat*}3
A_{\en}&=\M_{\en}^{d}(C_{\en}(L_{\en}+1))^{d(n+1)-1}
\end{alignat*}
and 
\begin{alignat*}3
B=A_{\en}R_Kh_K\sum_{g\in G(K/k)}\delta_{g}(K/k)^{-\mu_g}.
\end{alignat*}
Then as $X>0$ tends to infinity we have
\begin{alignat*}3
Z_{\en}(\IP^n(K/k),X)=
2^{-r_K(n+1)}\pi^{-s_K(n+1)}V_{\en}S_K(n)X^{d(n+1)}
+O(B X^{d(n+1)-1}\L_{\en}),
\end{alignat*}
where
\begin{alignat*}3
\L_{\en}&=\log\max\{2,2C_{\en}X\} \text{ if }(n,d)=(1,1)\text{ and }\L_{\en}=1 \text{ otherwise}
\end{alignat*}
and the implied constant in the $O$ depends only on $n$ and $d$.
\end{theorem}
With $k=K$ Theorem \ref{prop3}  yields a more general version of the Proposition in \cite{1}
with an explicit error term regarding the field $K$.
Still with $k=K$, let us choose the standard $\ALS$ with $N_v$ as in (\ref{Nvmaxnorm}) for all places $v$.
Then $\hen$ is just the Weil height on $\IP^n(K)$. Moreover 
$\Lamen(\D)=\sigma(\D)^{n+1}$ so that $\det \Lamen(\D)=(2^{-s_K}N(\D)\sqrt{|\Delta_K|})^{n+1}$ and therefore $V_{\en}^{fin}=1$.
Furthermore $V_{\en}^{inf}=\prod_{v \mid \infty}V_{v}=2^{r_K(n+1)}\pi^{s_K(n+1)}$ and thus $V_{\en}=2^{r_K(n+1)}\pi^{s_K(n+1)}$.
Hence we recover Schanuel's Theorem, but with an explicit 
error term with respect to the field. A more precise
version can be obtained by counting primitive points 
(over $\IQ$) for all subfields of $K$ (see \cite{WiThesis} Corollary 3.2).\\

Now back to the general case where $k$ is an arbitrary fixed subfield of $K$. Let us choose the $\ALS$ with $N_v$ as in (\ref{Nvmaxnorm})
if $v\nmid \infty$ and 
$N_v(\vz)=M(z_0x^n+z_1x^{n-1}+...+z_n)$ as in (2.7) of \cite{1} if $v \mid \infty$. Here $M$ denotes the Mahler measure.
The continuity of $M$ as a function of the coefficients was already shown by Mahler (see Lemma 1 in \cite{11}).
Masser and Vaaler have shown that the conditions $(i)$, $(ii)$ and $(iii)$ in Definition \ref{defALS} are satisfied and clearly $(iv)$
holds as well. Masser and Vaaler have also calculated $V_{\en}^{inf}=2^{r_K(n+1)}\pi^{s_K(n+1)}V_{\IR}(n)^{r_K}V_{\IC}(n)^{s_K}$
where $V_{\IR}(n)$ and $V_{\IC}(n)$ are certain rational numbers defined in \cite{1}. As in the previous example we have $V_{\en}^{fin}=1$
and therefore $V_{\en}=2^{r_K(n+1)}\pi^{s_K(n+1)}V_{\IR}(n)^{r_K}V_{\IC}(n)^{s_K}$.
Here Theorem \ref{prop3} counts the monic polynomials
$f=\alpha_0x^n+\alpha_1x^{n-1}+...+\alpha_n$ in $K[x]$ of degree at most $n$ whose coefficients $\alpha_0,\alpha_1,...,\alpha_n$
generate the whole field $K$ over $k$ and whose global absolute Mahler measure $M_0(f)=\hen(\alpha_0:...:\alpha_n)$ does not exceed $X$.
This adelic Lipschitz system will be used to deduce the main result in \cite{art4}.\\

In \cite{8} Roy and Thunder introduced
the quantity 
\begin{alignat*}3
\delta(K)=\underset{\alpha}\inf\{H(1,\alpha);K=\IQ(\alpha)\}.
\end{alignat*}
Generalizing this definition to extensions $K/k$
of number fields $k,K$
\begin{alignat*}3
\delta(K/k)=\underset{\alpha}\inf\{H(1,\alpha);K=k(\alpha)\}
\end{alignat*}
we can  give a simpler error term in Theorem \ref{prop3}.
Of course $\delta_1(K/k)=\delta(K/k)$ but we do not use this fact.
We define the integers
\begin{alignat*}3
\ggg=\max_{g\in G} g
\end{alignat*}
and
\begin{alignat}3
\label{mu}
\mu=m(e-\ggg)(n+1)-1.
\end{alignat}
Note that $1\leq \ggg \leq \max\{1,e/2\}$ and $\mu=\min_{g\in G}\mu_g\geq d(n+1)/2-1$.
We have the following 
\begin{corollary}\label{prop2}
Let $k,K$ be number fields with $k\subseteq K$ and $[K:k]=e$, $[k:\IQ]=\m$, $[K:\IQ]=d$. Let $\en$ be an adelic Lipschitz system of dimension $n$ 
on $K$ with associated constants $C_{\en},L_{\en},\M_{\en}$
and write
\begin{alignat*}3
A_{\en}&=\M_{\en}^{d}(C_{\en}(L_{\en}+1))^{d(n+1)-1}.
\end{alignat*}
Then as $X>0$ tends to infinity we have
\begin{alignat*}3
Z_{\en}(\IP^n(K/k),X)=
&2^{-r_K(n+1)}\pi^{-s_K(n+1)}V_{\en}S_K(n)X^{d(n+1)}\\
+&O(A_{\en} R_K h_K\delta(K/k)^{-\mu}X^{d(n+1)-1}\L_{\en})
\end{alignat*}
where 
\begin{alignat*}3
\L_{\en}&=\log\max\{2,2C_{\en}X\} \text{ if }(n,d)=(1,1)\text{ and }\L_{\en}=1 \text{ otherwise}
\end{alignat*}
and the implied constant in the $O$ depends only on $n$ and $d$.
\end{corollary}
To see that Theorem \ref{prop3} implies 
Corollary \ref{prop2} we need 
the following well-known argument. Since it will
be used also in the Section \ref{Estminima}, we give a proof here.
\begin{lemma}\label{lemmaprimele}
Let $F$ be a field of characteristic zero
and $L$ a finite extension of relative degree $e$ 
generated by  $\alpha_1,...,\alpha_t$. 
Then there are integers $0\leq m_1,...,m_t<e$
such that $F(\alpha)=L$ for
$\alpha=\sum_{j=1}^{t}m_j\alpha_j$.
\end{lemma}
\begin{proof}
It is well-known and easily seen (e.g. by induction on $t$) that for a polynomial 
$P(X_1,...,X_t)\in F[X_1,...,X_t]$
not identically zero with total degree $p$ we can 
find integers $m_1,...,m_t$ among $0,...,p$ such that 
$P(m_1,...,m_t)\neq 0$.
Now the case $e=1$ is trivial and so we may assume $e>1$.
Denote the conjugates 
of $\alpha_j$ over $F$ by $\alpha_j^{(i)}$ for 
$1\leq i\leq e$. We consider the polynomial
\begin{alignat}3
\label{konjpoly}
P(X_1,...,X_t)=
\prod_{i=2}^{e}
\left(\sum_{j=1}^{t}(\alpha_j^{(1)}-\alpha_j^{(i)})X_j\right).
\end{alignat}
Since $L=F(\alpha_1,...,\alpha_t)$
none of the factors $\sum_{j=1}^{t}(\alpha_j^{(1)}-\alpha_j^{(i)})X_j$
are zero and so $P$ is not identically zero
and of total degree $e-1$.
Using the observation of the beginning we get
integers $m_1,...,m_t$ with $0\leq m_j<e$ such that
$P(m_1,...,m_t)\neq 0$. But this implies
$\alpha=\sum_{j=1}^{t}m_j\alpha_j$ generates
$L$ over $F$.
\end{proof}

Now let us prove that Theorem \ref{prop3} implies
Corollary \ref{prop2}. We have to show that the error
term in the former is bounded above by the error
term in the latter. If $K=k$ then $\delta=\delta(K/k)=1$,
while $G(K/k)=\{1\}$ and
$\delta_1(K/k)=1$, $\mu_1=-1$. So we are done.
If $K\neq k$ then each $g$ in $G(K/k)$ satisfies
$g\leq \ggg$ and so $\mu_g\geq \mu$. Thus we have to compare
$\delta_g=\delta_g(K/k)$ with $\delta$.
Let $\alpha_1, \alpha_2$ be any 
numbers in $K$ such that $k(\alpha_1,\alpha_2)=K$.
By the previous lemma we deduce that there are rational
integers $0\leq m_1, m_2<e$ such that $\xi=m_1\alpha_1+m_2\alpha_2$
is primitive, so $K=k(\xi)$. Hence $\delta(K/k)\leq H(1,\xi)$.
On the other hand an easy calculation shows 
$H(1,\xi)\leq 2H(1,m_1, m_2)H(1,\alpha_1, \alpha_2)\leq 2eH(1,\alpha_1, \alpha_2)$.
Hence $\delta \leq 2e\delta_g$
for all $g$ in $G(K/k)$. This suffices to deduce
Corollary \ref{prop2} from Theorem \ref{prop3}.\\

\section{Preliminaries on counting}\label{1subsec1}
Recall that for a vector $\vx$ in $\IR^\Da$ we write $|\vx|$ 
for the euclidean length of $\vx$. The closed euclidean ball centered at $\vz$ with radius $r$
will be denoted by $B_{\vz}(r)$.
Let $\Lambda$ be a lattice of rank $\Da$ in $\IR^\Da$ then we define the \it successive minima \rm $\lambda_1(\Lambda),...,\lambda_\Da(\Lambda)$ of $\Lambda$ as the successive minima in the sense of Minkowski with respect
to the unit ball. That is
\begin{alignat*}3
\lambda_i=\inf \{\lambda;\,\lambda B_0(1)\cap \Lambda
\text{ contains $i$ linearly independent vectors}\}.
\end{alignat*}
By definition we have
\begin{alignat}3
0<\lambda_1\leq \lambda_2\leq...\leq \lambda_\Da<\infty.
\end{alignat}
Next we prove a simple lemma which will be used not only in this but also in Section \ref{Estminima}.
\begin{lemma}\label{lemma1.3}
Suppose $V$ is a subspace of $\IR^\Da$ of dimension
$i-1\geq 1$ and contains $i-1$ linearly independent
elements $v_1,...,v_{i-1}$ of $\Lambda$ with $|v_j|=\lambda_{j}$ for $1\leq j\leq i-1$.
Then any $v$ in $\Lambda$ not in $V$ satisfies
\begin{alignat*}3
|v|\geq \lambda_i.
\end{alignat*}
\end{lemma}
\begin{proof}
Suppose $v$ is in $\Lambda$ but not in $V$.
Then $v_1,...,v_{i-1},v$ are linearly independent.
Hence one of these vectors has length at least $\lambda_i$.
If $\lambda_{i-1}<\lambda_i$ the claim follows at once
since $|v_1|\leq ...\leq|v_{i-1}|=\lambda_{i-1}$.
Now let $p$ in $\{1,...,i\}$ be minimal with
$\lambda_p=\lambda_i$. If $p=1$ then the result is
clear from the definition of $\lambda_1$. If $p>1$
then $v_1,...,v_{p-1},v$ are linearly independent
and again we conclude one of these vectors has length at least $\lambda_p=\lambda_i$. But $v_1,...,v_{p-1}$ have 
length at most $\lambda_{p-1}<\lambda_i$, so
$|v|\geq \lambda_i$ as claimed.
\end{proof}
\begin{lemma}\label{minpowlatt}
Suppose $\Da=d(n+1)$ and $\Lambda=\Lambda_0^{n+1}$ for a
lattice $\Lambda_0$ of rank $d$ in $\IR^d$. Then the
successive minima of $\Lambda$ are given by
\begin{alignat*}3
\lambda_1(\Lambda_0),...,
\lambda_1(\Lambda_0),
\lambda_2(\Lambda_0),...,
\lambda_2(\Lambda_0),...,
\lambda_d(\Lambda_0),...,
\lambda_d(\Lambda_0)
\end{alignat*}
where each minimum is repeated $n+1$ times.
\end{lemma}
\begin{proof}
A typical minimum $\lambda_i(\Lambda_0)$ occurs
above in the positions $(i-1)(n+1)+1,...,i(n+1)$. Thus
it suffices to verify
\begin{alignat}3
\label{sucminineq1}
\lambda_{i(n+1)}({\Lambda_0}^{n+1})\leq 
\lambda_i(\Lambda_0)\leq  
\lambda_{(i-1)(n+1)+1}({\Lambda_0}^{n+1})
\end{alignat}
for $1\leq i\leq d$. For the first inequality
we note that there is a subspace $V_i$ in
$\IR^d$ of dimension $i$ containing $i$
linearly independent elements $v_1,...,v_i$ of $\Lambda_0$
with length $\lambda_1(\Lambda_0),...,\lambda_i(\Lambda_0)$.
Now $V_i^{n+1}$ in $\IR^{d(n+1)}$ of dimension
$i(n+1)$ contains $i(n+1)$ linearly independent elements
of $\Lambda_0^{n+1}$ like $(v_1,0,...,0)$ also with
length at most $\lambda_i(\Lambda_0)$. The first
inequality in (\ref{sucminineq1}) follows at once.\\
For the second inequality note that any $(i-1)(n+1)+1$ independent
points $w$ of $\Lambda_0^{n+1}$ cannot all lie in $V_{i-1}^{n+1}$.
So some $w$ has the form $w=(w_1,...,w_{n+1})$ with some $w_j$
not in $V_{i-1}$. By the previous lemma we see that
$|w|\geq |w_j|\geq \lambda_i(\Lambda_0)$ and the second inequality
is proved.
\end{proof}
To quantify the deficiency from
being orthogonal one defines the \it orthogonality
defect $\OD$ of a set of linearly independent
vectors \rm $v_1,...,v_\Da$ in $\IR^\Da$ as
\begin{alignat*}3
\OD(v_1,...,v_\Da)=\frac{|v_1|...|v_\Da|}{\det\Lambda}
\end{alignat*}
where $\Lambda$ is the lattice generated
by $v_1,...,v_\Da$.
By Hadamard's inequality $\OD(v_1,...,v_\Da)\geq 1$ with equality
if and only if the system of vectors is orthogonal.
When working with a lattice it is often convenient 
to have a basis $v_1,...,v_\Da$ of small orthogonality defect.
We define the \it orthogonality 
defect of the lattice \rm $\Lambda$ as
\begin{alignat*}3
\OD(\Lambda)=
\inf_{(v_1,...,v_\Da)}\frac{|v_1|...|v_\Da|}{\det\Lambda}
\end{alignat*}
where the infimum runs over all bases $(v_1,...,v_\Da)$
of $\Lambda$. Since $\Lambda$ is discrete the infimum
will be attained. 
Due to its importance it is worth to state Minkowski's Theorem
explicitly. Since we need only a special case
we do not give the full theorem (see \cite{18} p.218 Theorem V).
\begin{theorem}[(Minkowski's Second Theorem for balls)]\label{Mink2}
Let $\Lambda$ be a lattice in $\IR^\Da$ with successive
minima $\lambda_1,...,\lambda_\Da$. Then 
\begin{alignat*}3
\frac{2^\Da}{\Da!}\det \Lambda \leq \lambda_1...\lambda_\Da \Vol B_0(1)
\leq  2^\Da\det \Lambda
\end{alignat*}
where $\Vol B_0(1)=\frac{\pi^{\Da/2}}{\Gamma(\Da/2+1)}$.
\end{theorem}
\begin{proof}
For a proof we refer to \cite{18} p.205.
\end{proof}
By Minkowski's Second Theorem 
we obtain $n$ linearly independent vectors $u_1,...,u_\Da$
in $\Lambda$,
such that $|u_1|...|u_\Da|/\det\Lambda=\lambda_1...\lambda_\Da/\det\Lambda$ is bounded below and above in terms of $\Da$ only. Unfortunately these vectors usually fail to build a basis of the lattice but they can be used to construct a reduced basis.
We use the Mahler-Weyl basis reduction to prove the 
following bound:
\begin{lemma}\label{bouOD1}
Let $\Lambda$ be a lattice of rank $\Da>1$. Then
\begin{alignat*}3
\OD(\Lambda)\leq \frac{\Da^{\frac{3}{2}\Da}}{(2\pi)^{\frac{\Da}{2}}}.
\end{alignat*}
\end{lemma}
\begin{proof}
By Theorem \ref{Mink2}
\begin{alignat*}3
\lambda_1...\lambda_\Da \Vol B_0(1)\leq 2^\Da\det\Lambda.
\end{alignat*}
It is known from the definition of the $\lambda_i$ that 
there are linearly 
independent vectors $u_1,...,u_\Da$, such that $|u_i|=\lambda_i$
for $1\leq i \leq \Da$. Using a lemma of Mahler and Weyl
(\cite{18} Lemma 8 p.135) we obtain a basis $v_1,...,v_\Da$
of $\Lambda$ satisfying 
\begin{alignat*}3
|v_i|\leq \max\{|u_i|,\frac{1}{2}(|u_1|+...+|u_i|)\}
\leq \max\{1,\frac{i}{2}\}\lambda_i 
\end{alignat*}
for $1\leq i\leq \Da$.
Since $\Gamma(m+1)=m!$ and 
$\Gamma(m+1/2)=(m-1/2)(m-3/2)(m-5/2)...(1/2)\sqrt{\pi}$ for positive integers $m$,
we see that
$\Gamma(\frac{\Da}{2}+1)\leq (\frac{\Da}{2})^{\frac{\Da}{2}}$
provided $\Da\geq 2$. 
Using also $\Da!\leq \Da^{\Da-1}$ this yields
\begin{alignat*}3
\OD(\Lambda)\leq \frac{|v_1|...|v_\Da|}{\det\Lambda}
\leq \frac{\Da \Da!\Gamma(\frac{\Da}{2}+1)}{\pi^{\frac{\Da}{2}}}
\leq \frac{\Da^{\frac{3}{2}\Da}}{(2\pi)^{\frac{\Da}{2}}}
\end{alignat*}
and proves the statement.
\end{proof}

\section{The basic counting technique}\label{1subsec3}
Let $\Lambda$ be a lattice in $\IR^\Da$ of rank $\Da$.
A set $F$ is called a \it fundamental domain \rm of $\Lambda$
if there is a basis $v_1,...,v_\Da$ of $\Lambda$
such that
\begin{alignat*}3
F=[0,1)v_1+...+[0,1)v_\Da.
\end{alignat*}
Let $v_1,...,v_\Da$ be 
a basis of $\Lambda$ with corresponding fundamental domain
$F$. For a set $S$ in $\IR^\Da$ 
write $\T=\T_S(F)$ for the number of translates 
$F_v=F+v$ $(v \in \Lambda)$ by lattice points
having non-empty intersection with the boundary $\partial S$.
The following inequality is well-known  
but crucial. Therefore we state it as a lemma.
\begin{lemma}\label{latticeerrorboundary}
Suppose $S$ is measurable and bounded. Then
\begin{alignat}3
\label{lipapp}
||\Lambda \cap S|-\frac{\Vol S}{\det\Lambda}| \leq \T.
\end{alignat}
\end{lemma}
\begin{proof}
Clearly the translates $F_v=F+v$ $(v\in \Lambda)$
define a partition of $\IR^\Da$.
Moreover every $F_v$ contains exactly one lattice point
- namely $v$.
Denote by 
$\mathfrak{m}=\mathfrak{m}_S(F)$ the number of translates 
of $F$ by lattice points, which 
have empty intersection with the complement of $S$. 
In particular we have $\mathfrak{m} \leq |\Lambda \cap S|$.
Now suppose $v$ lies in $S$. So either $F_v$
lies in $S$ or $F_v$ contains a point of $S$ 
and a point of its complement.
But $F_v$ is convex and therefore
connected.
So if $F_v$ contains
a point of $S$ and a point of its complement then
it contains a point of the boundary $\partial S$.
Hence $|\Lambda \cap S|\leq \mathfrak{m}+\T$.
Now $\det\Lambda$ is the volume of $F_v$.
So the union of all translates $F_v$ lying in
$S$ has volume $\mathfrak{m} \det \Lambda$. And the union of all 
translates having non-empty intersection with $S$
has volume at most $(\mathfrak{m}+\T)\det \Lambda$.  
Thus we have proven the following inequalities:
\begin{alignat*}3
\mathfrak{m} &\leq |\Lambda \cap S| &\leq &\mathfrak{m}+\T,\\
\mathfrak{m} \det \Lambda &\leq \Vol S &\leq &(\mathfrak{m}+\T) \det \Lambda.
\end{alignat*}
Hence 
\begin{alignat*}3
||\Lambda \cap S|-\frac{\Vol S}{\det\Lambda}| \leq \T.
\end{alignat*}
\end{proof}
The inequality above explains why
the following proposition is crucial for the subsequent counting results of this section. 
\begin{proposition}[Masser]\label{PM1}
Assume $\Da>1$, let $\Lambda \subseteq \IR^\Da$ be a lattice and
let $\lambda_1,...,\lambda_\Da$ be the successive minima
of $\Lambda$ with respect to the unit ball.
Assume $S$ is a bounded subset of $\IR^\Da$ with 
boundary $\partial S$ in Lip$(\Da,1,\M,L)$. 
Let $v_1,...,v_\Da$ be a basis of $\Lambda$ with fundamental
domain $F$ and $\T_S(F)$ the number of translates
$F_v=F+v$ $(v \in \Lambda)$,
which have non-empty intersection with $\partial S$. 
Then for any natural number $Q$ we have
\begin{alignat*}3
\T_S(F) \leq \M   Q^{\Da-1} \prod_{i=1}^{\Da}{\left(\frac{\sqrt{\Da-1}\OD(v_1,...,v_\Da)L}{\lambda_iQ}+2\right)}.
\end{alignat*}
\end{proposition}
\begin{proof}
We certainly may assume that $S$ is not empty and therefore that $\partial S$ is not empty. 
Choose one of the parameterizing maps $\phi$ and split $I=[0,1]$ in $Q$ intervals of length $1/Q$.
Then $\phi(I^{\Da-1})$ splits in $Q^{\Da-1}$
subsets $\phi(C)$ where $C$ is a hypercube in $\IR^{\Da-1}$
of side $1/Q$. Due to the Lipschitz condition the distance between any two points in $\phi(C)$ does not exceed
$\frac{\sqrt{\Da-1}L}{Q}$.
Now $F$ is the fundamental domain corresponding to the given basis so $F=[0,1)v_1+...+[0,1)v_\Da$. We have to count the $v$ in $\Lambda$
such that $F_v$ meets $\partial S$. Thus $F_v$ meets one of
the $\phi(C)$ say in a point $\vx$. Writing $v=r_1v_1+...+r_\Da v_\Da$ for $r_1,...,r_\Da$ in $\IZ$, we see that there are 
$\vartheta_1,...,\vartheta_\Da$ in $[0,1)$ such that
\begin{alignat*}3
\vx=(r_1+ \vartheta_1)v_1+...+(r_\Da+\vartheta_\Da)v_\Da.
\end{alignat*}
We now show that there are not too many other $v'$ in $\Lambda$
such that $F_{v'}$ meets this same $\phi(C)$.
Let $\vx'$ be in $\phi(C)\cap F_{v'}$ then we 
get corresponding
$r_i', \vartheta_i'$.
To estimate the length of $\vx-\vx'$
write $\rho_i=r_i+\vartheta_i-(r_i'+\vartheta_i')$ for the
coefficient of the basis element $v_i$. 
Hence
\begin{alignat}3
\label{difflenght}
|\rho_1v_1+...+\rho_\Da v_\Da|=|\vx-\vx'|\leq \frac{\sqrt{\Da-1}L}{Q}.
\end{alignat}
After permuting the indices we may assume that $|v_i|\leq |v_{i+1}|$ and therefore $|v_i|\geq \lambda_i$. Now by Cramer's rule and the definition of $\OD(v_1,...,v_\Da)=\OD$ we get
\begin{alignat*}3
|\rho_i|=&|\frac{\det[v_1...\vx-\vx'...v_\Da]}
{\det[v_1...v_i...v_\Da]}|=\frac{|\det[v_1...\vx-\vx'...v_\Da]|}
{|v_1|...|v_i|...|v_\Da|}\OD.
\end{alignat*}
Now we apply Hadamard's inequality to obtain the upper bound 
\begin{alignat*}3
\frac{|v_1|...|\vx-\vx'|...|v_\Da|}{|v_1|...|v_i|...|v_\Da|}\OD
=\frac{|\vx-\vx'|}{|v_i|}\OD \leq \frac{|\vx-\vx'|}{\lambda_i}\OD.
\end{alignat*}
Due to (\ref{difflenght}) the latter is 
\begin{alignat*}3
\leq &\frac{\sqrt{\Da-1}\OD L}{\lambda_i Q}.
\end{alignat*}
Notice that $|\vartheta_i-\vartheta_i'| < 1$ therefore all
the $r_i$ lie in an interval of length
\begin{alignat*}3
\frac{\sqrt{\Da-1}\OD L}{\lambda_i Q}+1. 
\end{alignat*}
So the number of $(r_1,...,r_\Da)$ is at most
\begin{alignat*}3
\prod_{i=1}^{\Da}
\left([\frac{\sqrt{\Da-1}\OD L}{\lambda_iQ}]+2\right),
\end{alignat*}
provided there are at least two of them. However, it is trivially
true if there is just one of them.
On recalling that we have $\M$ parameterizing maps and
$Q^{\Da-1}$ subsets $\phi(C)$ for each map 
we get the desired
upper bound for the number of translates having
non-empty intersection with the boundary of $S$. 
\end{proof}

The Proposition \ref{PM1} leads to an explicit version
of Lemma 2 \cite{1}. 
\begin{corollary}\label{TMV1}
Let $S$ be a bounded set in $\IR^\Da$ such that
the boundary $\partial S$ of $S$ is in Lip$(\Da,1,\M,L)$.
Let $\Lambda$ be a lattice in $\IR^\Da$.
Then $S$ is measurable and moreover 
\begin{alignat}3
\label{ineqcor1}
||S\cap\Lambda|-\frac{\Vol S}{\det\Lambda}|
\leq 3^\Da \M \left(\frac{\sqrt{\Da}\OD(\Lambda)L}{\lambda_1}+1\right)^{\Da-1}.
\end{alignat}
\end{corollary}
\begin{proof}
For $\Da=1$ the set $S$ is a union of at most $\M$ intervals (or
even single points) in which case the statement is trivial.
So we may assume $\Da>1$. For the measurability
we refer to \cite{20} Satz 7 p.294.
To prove the second statement we choose a basis with minimal
orthogonality defect. Thanks to (\ref{lipapp})
it suffices to estimate $\T$ corresponding to this basis. 
Using Proposition \ref{PM1} we see that $\T$ is bounded above by $\M Q^{\Da-1}(\frac{\sqrt{\Da-1}\OD(\Lambda)L}{\lambda_1Q}+2)^\Da$.
Now let us choose
$Q=[\frac{\sqrt{\Da}\OD(\Lambda)L}{\lambda_1}]+1$.
This leads straightforwardly to
\begin{alignat*}3
\T\leq 
3^\Da \M \left(\frac{\sqrt{\Da}\OD(\Lambda)L}{\lambda_1}+1\right)^{\Da-1}
\end{alignat*}
and the theorem is proved.
\end{proof}
For our application in Section \ref{subseccounting}
we need a more precise result which takes
into account not only the first but
also the other minima. 
\begin{theorem}\label{TWi1}
Let $\Lambda$ be a lattice in $\IR^\Da$
with successive minima (with respect to the
unit ball) $\lambda_1,...,\lambda_\Da$.
Let $S$ be a bounded set in $\IR^\Da$ such that
the boundary $\partial S$ of $S$ is in Lip$(\Da,1,M,L)$.
Then $S$ is measurable and moreover
\begin{alignat*}3
||S\cap\Lambda|-\frac{\Vol S}{\det \Lambda}|
\leq c_0(\Da)\M\max_{0\leq i<\Da}\frac{L^i}{\lambda_1...\lambda_i}.
\end{alignat*}
For $i=0$ the expression in the maximum is to be understood
as $1$. Furthermore one can choose $c_0(\Da)=\Da^{3\Da^2/2}$.
\end{theorem}
\begin{proof}
For the measurability see Corollary \ref{TMV1}. 
Since the case $\Da=1$ is straightforward we assume $\Da>1$.
As in the proof of Corollary \ref{TMV1} it suffices
to estimate $\T$ corresponding to a basis with
minimal orthogonality defect. 
To simplify notation we write $\kappa$ for 
$\sqrt{\Da-1}\OD(\Lambda)$.
It is convenient to distinguish two cases:\\

$(1)$ $L<\lambda_\Da:$\\
We use Proposition \ref{PM1} with $Q=1$. We estimate
the $\Da$-th term of the product by $\kappa+2$. So
\begin{alignat*}3
\T\leq \M(\kappa+2)
\prod_{i=1}^{\Da-1}\left(\frac{\kappa L}{\lambda_i}+2\right)
&\leq \M(\kappa+2)
\prod_{i=1}^{\Da-1}(\kappa+2)\left(\frac{L}{\lambda_i}+1\right)\\
&= \M(\kappa+2)^\Da
\prod_{i=1}^{\Da-1}\left(\frac{L}{\lambda_i}+1\right).
\end{alignat*}
Now we expand the remaining product and estimate each
of the $2^{\Da-1}$ terms in the resulting sum by 
$\max_{0\leq i<\Da}\frac{L^i}{\lambda_1...\lambda_i}$.
Hence
\begin{alignat}3
\label{ineqTwi1}
\T\leq \M(\kappa+2)^\Da 2^{\Da-1}
\max_{0\leq i<\Da}\frac{L^i}{\lambda_1...\lambda_i}.
\end{alignat}
Next we use Lemma \ref{bouOD1} and recall that
$\Da>1$ to estimate 
\begin{alignat*}3
\kappa+2\leq \frac{\sqrt{\Da-1}\Da^{3\Da/2}}{(2\pi)^{\Da/2}}+2\leq \frac{1}{2\pi}\Da^{3\Da/2}+\frac{1}{4}\Da^{3\Da/2}<\frac{1}{2}\Da^{3\Da/2}.
\end{alignat*}
Hence
\begin{alignat*}3
\T\leq \M \Da^{3\Da^2/2}
\max_{0\leq i<\Da}\frac{L^i}{\lambda_1...\lambda_i},
\end{alignat*}
which proves the theorem in the first case.\\

$(2)$ $L\geq \lambda_\Da:$\\ 
Note that in particular $L>0$.
Here we choose $Q=[\frac{L}{\lambda_\Da}]+1$
and we get
\begin{alignat*}3
\T\leq \frac{\M}{Q}\prod_{i=1}^{\Da}\left(\frac{\kappa L}{\lambda_i}+2Q\right)
&\leq \frac{\M\lambda_\Da}{L}
\prod_{i=1}^{\Da}\left(\frac{(\kappa+2) L}{\lambda_i}+2\right)\\
&\leq \M(\kappa+4)^\Da\frac{L^{\Da-1}}{\lambda_1...\lambda_{\Da-1}}\\
&\leq \M 2^\Da(\kappa+2)^\Da\frac{L^{\Da-1}}{\lambda_1...\lambda_{\Da-1}}
\end{alignat*}
where this last $\frac{L^{\Da-1}}{\lambda_1...\lambda_{\Da-1}}$
is now the maximum term in (\ref{ineqTwi1}).
We have already seen that (for $\Da>1$) 
$\kappa+2\leq 2^{-1}\Da^{3\Da/2}$ and so the result drops out.
\end{proof}
Theorem \ref{TWi1} can be considered as a version of
Schmidt's Theorem on p.15 in \cite{7} with different
and probably weaker 
conditions on the set.

\section{The basic set}\label{thebasicset}
Recall that $K$ is a number field of degree $d$ with $r$ real and $s$ pairs of complex conjugate
embeddings. Recall also the basic notation of an adelic Lipschitz system $\en$ on $K$
of dimension $n$.
The constants $C_{\en}, \M_{\en}, L_{\en}$ will be abbreviated
to $C, M, L$.
Lemma \ref{lemma2.8} and Lemma \ref{ZKHX} of the following sections have much in common with 
Lemma 3 and Lemma 4 of \cite{1}. For the convenience of the reader we tried to keep the notation of 
\cite{1} whenever possible. 
So let $q=r+s-1$, $\Sigma$ the hyperplane in $\IR^{q+1}$
defined by $x_1+...+x_{q+1}=0$ and 
$\vdelta=(d_1,...,d_{q+1})$ with 
$d_i=1$ for $1\leq i \leq r$ and $d_i=2$ 
for $r+1\leq i \leq r+s=q+1$.
The map $l(\eta)=(d_1\log|\sigma_1(\eta)|,...,d_{q+1}\log|\sigma_{q+1}(\eta)|)$
sends $K^*$ to $\IR^{q+1}$. For $q>0$ the image of the unit group
$\IU=\Oseen_K^*$ under $l$ is a lattice in $\Sigma$ with
determinant $\sqrt{q+1}R_K$.\\

Let $F$ be a bounded set in $\Sigma$ and for
real, positive $T$ let $F(T)$ be the vector sum
\begin{alignat}3
\label{vecsum1}
F(T)=F+\vdelta(-\infty,\log T].
\end{alignat}
We denote by $\exp$ the diagonal exponential map
from $\IR^{q+1}$ to $[0,\infty)^{q+1}$.
We have $r+s$ Lipschitz distance functions $N_1,...,N_{q+1}$
one for each factor of $\IR^r\times \IC^s$. We
use variables ${\bf z}_1,...,{\bf z}_{q+1}$ with ${\bf z}_i$
in $\IR^{d_i(n+1)}$. Now we define $S_F(T)$ in $\IR^D$
for $D=\sum_{i=1}^{q+1}d_i(n+1)=d(n+1)$ as the set of  all
${\bf z}_1,...,{\bf z}_{q+1}$ such that
\begin{alignat}3
\label{inkl1}
(N_1({\bf z}_1)^{d_1},...,N_{q+1}({\bf z}_{q+1})^{d_{q+1}}) \in \exp(F(T)).
\end{alignat}

\section{On Lipschitz parameterizability}\label{lpb}
As we have seen  in Section \ref{1subsec3} one can give good estimates for
the number of lattice points in a bounded set under rather mild conditions on the set such as the
Lipschitz parameterizability of the boundary.
As shown by Masser and Vaaler in \cite{1} Lemma 3
the condition $(iii)$ in Section \ref{secALH}
implies that the set
$S_F(T)$ has Lipschitz parameterizable boundary of
co-dimension one.
To see the dependence on $F,L,\M$ for the Lipschitz constant
we need an explicit
(up to dependence on $n,d$) version of this Lemma 3.
This can be done in a relatively straightforward manner and might be a bit tedious for the reader.
However, we have carried out this checking very carefully and to the
best of the author's knowledge this is the first detailed account of such matters in the 
literature, published and unpublished. But in order not to distract the reader too much from the basic line we
postpone the proof to the Appendix.
\begin{lemma}\label{lemma2.8}
Suppose $q\geq 1$ and let $F$ be a set in $\Sigma$ such that
$\partial F$ is in Lip$(q+1,2,\M',L')$ and moreover assume 
$F$ lies in $B_0(r_F)$. Then 
$\partial S_{F}(1)$ is in Lip$(D,1,\widetilde{\M},\widetilde{L})$
where one can choose
\begin{alignat*}3
\widetilde{\M}&=(\M'+1)\M^{q+1}\\
\widetilde{L}&=3\sqrt{D}(L'+r_F+1)
\exp(\sqrt{q}(L'+r_F))(L+C_{\en}^{inf}).
\end{alignat*}
\end{lemma}
\begin{proof}
See Appendix.
\end{proof}
Notice that for $q=0$ the boundary of $S_F(1)$ is nothing
but the set defined in $(iii)$ Section \ref{secALH} (for $v\mid\infty$)
and so in this case we have $\partial S_F(1)$ lies in
Lip$(D,1,M,L)$.\\

In our first application $F$ will have the form
\begin{alignat}3
\label{Flpb2}
[0,1)v_1+...+[0,1)v_{q}
\end{alignat}
for $v_1,...,v_q$ in $\IR^{q+1}$ with $|v_1|,...,|v_q|<1$.
It is easy to see that $\partial F$ is Lipschitz parameterizable;
a typical boundary point has the form $x_1v_1+...+x_qv_q$
with some $x_i=0$ or $1$, so for example if $i=q$
then this expression gives a parameterization on the variables
$x_1,...,x_{q-1}$. We find in this way that $\partial F$
is in Lip$(q+1,2,2q,q-1)$.

\section{Schmidt's partition method}
\label{schmidtstrick}

First suppose $q>0$. Recall the standard logarithmic map $l$ from
$K^*$ to $\IR^{q+1}$ (see Section \ref{thebasicset}). 
We choose $F$ as a fundamental domain of the unit
lattice $l(\IU)$ 
\begin{alignat*}3 
F=[0,1)u_1+...+[0,1)u_q
\end{alignat*}
where
$U=(u_1,...,u_q)$ is a basis of $l(\IU)$.
A major step of the proof is the counting of
lattice points in the set $S_F(T)$. This will be carried out 
with the help of Theorem \ref{TWi1}. But here the relevant 
Lipschitz constants may depend on the units in a fatal way.
In fact $F$ has volume $\sqrt{q+1}R_K$ and so if we are unlucky then it might not lie in a ball of radius much smaller than $R_K$. Thus $\exp(F)$ might not lie in a ball of radius much smaller than $\exp(R_K)$. This might introduce Lipschitz constants of this size and consequently the error terms in the counting could be this large. That however is far from what we claim in Theorem \ref{prop3}. 
And such an exponential dependence on $R_K$ would be disastrous for the summation techniques in the main application
following in \cite{art2}.
To overcome this problem we extend an idea of
Schmidt \cite{14} from the real-quadratic case $d=2$ to
arbitrary $d$ (see also \cite{7} for $d>2$).\\

Let us carry out the details. First we define
the $q+1$ natural numbers
\begin{alignat}3
\label{defnj}
n_j&=[|u_j|]+1 \quad (1\leq j \leq q),\\
\label{deft} 
t&=n_1...n_q. 
\end{alignat}
Let $Q=|\{\beta \in \Qbar;[\IQ(\alpha):\IQ]\leq d, \log H(1,\alpha)\leq 1\}|$. If $\alpha$ of degree at most $d$ is neither zero
nor a root of unity then the $Q+1$ numbers $1,\alpha,...,\alpha^Q$ are pairwise distinct and therefore $\log H(1,\alpha^Q)>1$, so
\begin{alignat*}3 
\log H(1,\alpha)> Q^{-1}.
\end{alignat*}
We take $\alpha=\eta_j$ for $l(\eta_j)=u_j$ to deduce
\begin{alignat*}3 
\exp(d/Q)\leq H(1,\eta_j)^d=\prod_{i=1}^{q+1}\max\{1,|\sigma_i(\eta_j)|^{d_i}\}.
\end{alignat*}
It follows that $|\sigma_i(\eta_j)|\geq \exp(1/Q)$ for some
$i$. Thus 
\begin{alignat*}3
|u_j|^2=\sum_{k=1}^{q+1}d_k^2\log^2|\sigma_k(\eta_j)|\geq (1/Q)^2
\end{alignat*}
and so
\begin{alignat*}3 
|u_j|\geq 1/Q > 0,
\end{alignat*}
where $Q$ depends only on $d$.
The inequality above implies 
$[|u_j|]+1\leq (1+Q)|u_j|$.
Recalling the definition of the orthogonality defect
$\OD(U)$ of $U$ and not forgetting that 
$\det l(\IU)=\sqrt{q+1}R_K$ yields
\begin{alignat*}3
\sqrt{q+1}R_K<t \leq (1+Q)^q \OD(U) \sqrt{q+1}R_K.
\end{alignat*}
Now we choose a reduced basis $U$ so that according to Lemma \ref{bouOD1} we have 
in particular $\OD(U)\leq d^{2d}$, provided $q>1$. But the latter inequality trivially remains true for $q=1$.
Hence there is a constant $c_d$ depending only on $d$ with
\begin{alignat}3
\label{tAbsch}
R_K<t\leq c_d R_K.
\end{alignat}
We define
\begin{alignat}3
\label{F}
F({\bf i})=i_1\frac{u_1}{n_1}+...+i_q\frac{u_q}{n_q}+
[0,1)\frac{u_1}{n_1}+...+[0,1)\frac{u_q}{n_q}
\end{alignat}
with ${\bf i}=(i_1,...,i_{q})$ for $0\leq i_j < n_j$ ($1\leq j \leq q$). Then the 
partition $F=\bigcup_{\bf{i}}F(\bf{i})$ leads to a partition
\begin{alignat}3
\label{partSF1}
S_F(T)=\bigcup_{\bf{i}}S_{F({\bf i})}(T)
\end{alignat}
in $t$ subsets. For each of these $t$ vectors
${\bf i}$ we define a translation $\tr$
on $\IR^{q+1}$ by 
\begin{alignat*}3
\tr(x)=x-\sum_{j=1}^{q}\frac{i_ju_j}{n_j}.
\end{alignat*}
This translation sends $\Sigma$ to $\Sigma$ and
$F({\bf i})$ to $F({\bf 0})$.
It has an exponential counterpart $\etr$ defined by
$\etr(\exp(x))=\exp(\tr(x))$ and this takes the form
\begin{alignat*}3
\etr(X_1,...,X_{q+1})=
(\gamma_1^{d_1}X_1,...,\gamma_{q+1}^{d_{q+1}}X_{q+1})
\end{alignat*}
for positive real $\gamma_1,...,\gamma_{q+1}$, depending on
${\bf i}$, with
\begin{alignat}3
\label{prodgammai}
\gamma_1^{d_1}...\gamma_{q+1}^{d_{q+1}}=1.
\end{alignat}
We define the automorphism $\ti$ of $\IR^D$ by
\begin{alignat}3
\label{ti}
\ti({\bf z}_1,...,{ \bf z}_{q+1})=(\gamma_1{\bf z}_1,...,\gamma_{q+1}{\bf z}_{q+1}),
\end{alignat}
so that
\begin{alignat}3
\label{dettaui}
\det \ti=1.
\end{alignat}
Now 
\begin{alignat*}3
\etr(\exp(F({\bf i})(T)))=\exp(\tr(F({\bf i})(T)))
=\exp(F({\bf 0})(T))
\end{alignat*}
and so (\ref{inkl1}) together with $(ii)$ of Section \ref{secALH} gives
\begin{alignat}3
\label{tr1}
\ti S_{F({\bf i})}(T)=S_{F({\bf 0})}(T).
\end{alignat}
The identity
\begin{alignat}3
\label{homexp1}
S_{F}(T)=T S_{F}(1)
\end{alignat}
holds for any $F$ in $\Sigma$ whatsoever
and in particular
\begin{alignat}3
\label{homexp2}
S_{F({\bf 0})}(T)=T S_{F({\bf 0})}(1).
\end{alignat}
Thanks to (\ref{defnj}) and the triangle inequality,
$|\theta_1\frac{u_1}{n_1}+...+\theta_q\frac{u_q}{n_q}|\leq q$
holds for any $\theta_j\in [0,1)$.
From the definition of $F({\bf 0})$ and $S_{F({\bf 0})}$
it follows that
\begin{alignat}3
\label{OM1}
S_{F({\bf 0})}(1)\subseteq\{({\bf z}_{1},...,{\bf z}_{q+1});
N_i({\bf z}_{i})^{d_i}\leq \exp(q)  \text{ for }1\leq i\leq q+1\}.
\end{alignat}
On recalling the definition (\ref{defcinf}) of $C^{inf}_{\en}$ the above inclusion together with
(\ref{homexp2}) yields
\begin{alignat}3
\label{SF0supset}
S_{F({\bf 0})}(T)\subseteq B_0(\d_1T)
\end{alignat}
where $\d_1=\sqrt{d(n+1)}C^{inf}_{\en}\exp(q)$
and $B_0(\d_1T)$ denotes the euclidean ball
centered at the origin with radius $\d_1T$.\\
From now on let $\bf{i}$ be fixed so that we may drop
the index and write $\tau$.
The ${\bf z}_i$ lie in $\IR^{n+1}$ or $\IC^{n+1}$.
By abuse of notation we temporarily set $n=0$ so that we may interpret these vectors for
a moment as numbers in $\IR$ or $\IC$. Then
the right hand side of (\ref{ti}) defines an automorphism 
of $\IR^d$, say $p_{\tau}$ with
\begin{alignat}3
\label{dettau0}
\det p_{\tau}=1.
\end{alignat}
Notice that for a set $X$ in $\IR^d$ one has
$\tau(X^{n+1})=(p_{\tau}(X))^{n+1}$ in $\IR^{d(n+1)}=\IR^D$.
However, it will be more convenient to
write $\ta_0$ for $p_{\tau}$, just as the $\sigma$ in
(\ref{sigd}) is simply the $\sigma$ in $(\ref{sigD})$ with $n=0$.\\

Now suppose $q=0$. In this case the only units are roots of
unity and we set $F=\v0$. Here we may apply the counting principles of Section \ref{1subsec3} to the set $S_F(T)$ directly
without running into the difficulty of getting huge
Lipschitz constants. In order to treat this rather
easy case simultaneously with the more interesting case
$q>0$ it will be convenient to define the set of the vectors $\bf{i}$
as the set $\{\v0\}$ consisting only of the single vector
$\v0=(0)$ and we set $t=1$. Then we define $S_{F({\bf i})}(T)=S_{F({\bf 0})}(T)=S_{F}(T)$ and moreover
$\ti=\tau_{\bf{0}}$ is the identity automorphism. Hence an expression like $\bigcup_{\bf{i}}S_{F({\bf i})}(T)$ is to be understood
as $S_F(T)$. With these conventions
(\ref{tAbsch}), (\ref{partSF1}) and also (\ref{tr1}),
(\ref{homexp1}), (\ref{homexp2}), (\ref{OM1}), (\ref{SF0supset}) and (\ref{dettau0}) remain valid.

\section{Estimates for the minima}\label{Estminima}
We define the non-zero ideal $\C_0$ by
\begin{alignat}3
\label{C0}
\C_0=\prod_{v\nmid \infty}\pw_v^{-\frac{d_v\log c_v}{\log N \pw_v}}
\end{alignat}
with $c_v$ as in (\ref{defcfin}).
Thus $|\C_0|_v=c_v$ and
\begin{alignat}3
\label{NC0}
N\C_0=(C_{\en}^{fin})^d.
\end{alignat}
Let $\D\neq 0$ be a fractional ideal.
Clearly $|\alpha|_v\leq |\C_0^{-1}\D|_v$ for all
non-archimedean $v$ is equivalent to $\alpha \in \C_0^{-1}\D$.
By (\ref{Nineq1}) we conclude
\begin{alignat}3
\label{OG}
\Lamen(\D) \subseteq \sigma(\C_0^{-1}\D)^{n+1}.
\end{alignat}
Since $\en$ is fixed we can omit the index
and simply write $\Lam(\D)$ for $\Lamen(\D)$.
Certainly $\g0$ is a lattice  
in $\IR^d$. For each $\D$ we choose linearly independent 
vectors 
\begin{alignat*}3
v_1=\ta_0\sigma(\theta_1),...,v_d=\ta_0\sigma(\theta_d)
\end{alignat*}
of the lattice $\g0$ with 
\begin{alignat}3
\label{visuccmin}
|v_i|&=\lambda_i(\g0) \qquad (1\leq i\leq d)
\end{alignat}
for the successive minima.
Since $v_1,...,v_d$ are $\IR$-linearly independent,
$\ta_0^{-1}v_1,...,\ta_0^{-1}v_d$ are also $\IR$-linearly 
independent. Hence $\theta_1,...,\theta_d$ are $\IQ$-linearly 
independent and therefore $\frac{\theta_1}{\theta_1},...,\frac{\theta_d}{\theta_1}$
are $\IQ$-linearly independent. Now $[K:\IQ]=d$ implies  $K=\IQ(\frac{\theta_1}{\theta_1},...,\frac{\theta_d}{\theta_1})
=k(\frac{\theta_1}{\theta_1},...,\frac{\theta_d}{\theta_1})$
and this allows the following definition.
\begin{definition}\label{defl}
Let  $l \in \{1,...,d\}$ be minimal with 
$K=k(\frac{\theta_1}{\theta_1},...,\frac{\theta_{l}}{\theta_1})$.
\end{definition}
In principle $l$ depends on $k$, on the lattice
$\g0$ and on the choice of $v_1,...,v_d$.
So it depends on $k$, on
$\ta_0$ and  on $\C_0$, $\D$. But $\ta_0=\ta_0(\i)$ itself depends 
on $\i$ and on the basis $U$ of the unit lattice.
However, $k$, $\C_0$ and the choice of $U$ are fixed and for every
$\g0$ the choice of $v_1,...,v_d$ is fixed also
such that $l=l(\i,\D)$ depends only on the ideal
$\D$ and on the vector $\i$. Moreover we have the
following statement
which for $k=\IQ$ is Lemma 2.1 of \cite{7}. 
\begin{lemma}\label{lemma2.2}
We have
\begin{alignat*}3
l\leq \left[\frac{d}{2}\right]+1.
\end{alignat*}
\end{lemma}
\begin{proof}
Assume the statement is false then there exists a
proper subfield $K_0$ of $K$ containing the $[\frac{d}{2}]+1$ $\IQ$-linearly independent numbers $\frac{\theta_i}{\theta_1}$ for $1\leq i\leq [\frac{d}{2}]+1$. But $[K_0:\IQ]\leq d/2$ and so $K_0$ contains no more than $d/2$ $\IQ$-linearly 
independent numbers contradicting the fact $[\frac{d}{2}]+1>d/2$.
\end{proof}

We abbreviate 
\begin{alignat}3
\label{deflami1}
\lambda_i=\lambda_i(\g0)
\end{alignat}
for $1\leq i\leq d$.
\begin{lemma}\label{lemma2.3}
Assume $a \in \{1,...,d\}$ and $\mu_1,...,\mu_a$ in $\IR$
with $\mu_a\neq 0$ are such that $w=\mu_1v_1+...+\mu_av_a$ lies in $\g0$. 
Then we have 
\begin{alignat*}3
|w| \geq \lambda_a.
\end{alignat*}
\end{lemma}
\begin{proof}
For $a=1$ it is clear. For $a>1$ we apply Lemma \ref{lemma1.3}
of Section \ref{1subsec1} with $V=\IR v_1+...+\IR v_{a-1}$.
\end{proof}

\begin{lemma}\label{lemma2.4}
Assume $l\geq 2$, and let $\omega_0,...,\omega_n$ in $K$
be not all zero with $k(\omega_0:...:\omega_n)=K$.
Then not all of the $\omega_0,...,\omega_n$ are in
$k\theta_1+...+k\theta_{l-1}$.
\end{lemma}
\begin{proof}
Set $K_0=
k(\frac{\theta_1}{\theta_1},...,\frac{\theta_{l-1}}{\theta_1})$.
By definition of $l$ we have $K_0\subsetneq K$.
Let $a,b$ be in  $\{0,...,n\}$ with $\omega_b \neq 0$.
Suppose $\omega_a, \omega_b$ are in $k\theta_1+...+k\theta_{l-1}$. Then there are $\alpha_j,\beta_j$ ($1\leq j\leq l-1$) in $k$ 
such that
\begin{alignat*}3
\frac{\omega_a}{\omega_b}=
\frac{\sum_{j=1}^{l-1}\alpha_j\theta_j}
{\sum_{j=1}^{l-1}\beta_j\theta_j}=
\frac{\sum_{j=1}^{l-1}\alpha_j\frac{\theta_j}{\theta_1}}
{\sum_{j=1}^{l-1}\beta_j\frac{\theta_j}{\theta_1}}.
\end{alignat*}
But numerator and denominator of the last fraction are in $K_0$
and so $\frac{\omega_a}{\omega_b}$ is in $K_0$.
So if all $\omega_0,...,\omega_n$ are in $k\theta_1+...+k\theta_{l-1}$ then 
$k(\omega_0:...:\omega_n)\subseteq K_0$ - a contradiction.
\end{proof}

\begin{lemma}\label{lemma2.5}
Let $ \omega_0,...,\omega_n $ be in $ \C_0^{-1}\D$ not all zero
with $k(\omega_0:...:\omega_n)=K$. Then for  $v=(\ta_0\sigma\omega_0,...,\ta_0\sigma\omega_n)$
in $\IR^D$
we have
\begin{alignat*}3
|v| \geq \lambda_l.
\end{alignat*}
\end{lemma}
\begin{proof}
Each of the $\ta_0\sigma\omega_0$,...,
$\ta_0\sigma\omega_n$ lies in the lattice $\g0$. 
The sublattice generated by $v_1,...,v_d$ has finite index
in $\g0$. Hence there are $\mu_j^{(i)} \in \IQ$ such that 
\begin{alignat*}3
v=\left(\sum_{j=1}^{d}\mu_j^{(0)}v_j,...,\sum_{j=1}^{d}\mu_j^{(n)}v_j\right).
\end{alignat*}
Lemma \ref{lemma2.4} and the condition 
$K=k(\omega_0:...:\omega_n)$ imply at least one
of the numbers
$\mu_j^{(i)}$ for $l\leq j\leq d$, $0\leq i \leq n$ 
is non-zero and so the result follows by Lemma \ref{lemma2.3}.
\end{proof}
\begin{lemma}\label{lemma2.6}
If $l\geq 2$ then 
\begin{alignat}3
\label{case2}
\frac{l-1}{m}\leq [k\left(\frac{\theta_1}{\theta_1},...,\frac{\theta_{l-1}}{\theta_1}\right):k]
\leq \max\{1,e/2\}.
\end{alignat}
\end{lemma}
\begin{proof}
The $l-1$ numbers $\frac{\theta_1}{\theta_1},...,\frac{\theta_{l-1}}{\theta_1}$
are $\IQ$-linearly independent. Hence 
$[K_0:\IQ]\geq l-1$ for $K_0=k(\frac{\theta_1}{\theta_1},...,\frac{\theta_{l-1}}{\theta_1})$. The first inequality follows at once, since $m=[k:\IQ]$.
But the second one follows immediately from the
definition of $l$ since $[K:k]=e$.
\end{proof}
\begin{lemma}\label{lemma2.7}
We have
\begin{alignat*}3
\lambda_1 &\geq \sqrt{d/2}(C_{\en}^{fin})^{-1} N(\D)^{\frac{1}{d}}.
\end{alignat*}
Moreover with $K_0=k(\frac{\theta_1}{\theta_1},...,\frac{\theta_{l-1}}{\theta_1})$ if $l\geq 2$ and $K_0=k$ if $l=1$ and $g=[K_0:k]\in G(K/k)$ one has 
\begin{alignat*}3
\lambda_l &\geq \frac{1}{\sqrt{2}ed}(C_{\en}^{fin})^{-1} N(\D)^{\frac{1}{d}}\delta_g(K/k).
\end{alignat*}
\end{lemma} 
\begin{proof}
For the first statement observe that
by definition  
\begin{alignat*}3
\tau\sigma\alpha=
(\gamma_1\sigma_1\alpha,...,\gamma_{q+1}\sigma_{q+1}\alpha).
\end{alignat*}
So the squared length of an element 
$\tau\sigma\alpha$ of $\g0$ is
\begin{alignat*}3
\sum_{i=1}^{q+1}|\gamma_i\sigma_i\alpha|^2\geq
\frac{1}{2}\sum_{i=1}^{q+1}d_i|\gamma_i\sigma_i\alpha|^2.
\end{alignat*}
Next we use the inequality between the arithmetic and geometric mean to deduce that this is
at least
\begin{alignat*}3
(d/2)\prod_{i=1}^{q+1}|\gamma_i\sigma_i\alpha|^{2d_i/d}.
\end{alignat*}
By (\ref{prodgammai}) we see that the latter is 
$(d/2)\prod_{i=1}^{q+1}|\sigma_i\alpha|^{2d_i/d}$.
Here $\prod_{i=1}^{q+1}|\sigma_i\alpha|^{d_i}$
is the absolute value of the norm of $\alpha$ from
$K$ to $\IQ$ which is at least $N{\C_0}^{-1}\D$
provided $\alpha\neq 0$. Recalling (\ref{NC0}) we
see that $N{\C_0}^{-1}\D=(C_{\en}^{fin})^{-d}N\D$
which leads to the first statement.\\ 

Now let us prove the second estimate.
First note that $l=1$ is equivalent to $K=k$.
Thus $l=1$ implies $k=K$, $g=1$, $\delta_g(K/k)=1$ and so the claim follows from
the first statement. Next suppose $l>1$.
We apply Lemma \ref{lemmaprimele} twice
to obtain a primitive
element $\beta=\sum_{i=1}^{l}m_i\frac{\theta_i}{\theta_1}$
for the extension $K/k$
where $m_i$ are in $\IZ$ and $0\leq m_i<e$ ($1\leq i \leq l$).
And once more to get 
a primitive element $\alpha=\sum_{i=1}^{l-1}m_i'\frac{\theta_i}{\theta_1}$
for the extension 
$k(\frac{\theta_1}{\theta_1},...,\frac{\theta_{l-1}}{\theta_1})/k$
with 
$m_1',...,m_{l-1}'$
in $\IZ$ and $0\leq m_i'<e$ 
($1\leq i \leq l-1$).
So $k(\alpha,\beta)=K$ and $[k(\alpha):k]=g$. Using the product formula we get
\begin{alignat*}3
\delta_g(K/k)^d\leq H(1,\alpha,\beta)^d=
&\prod_{v\nmid \infty}\max\{|\theta_1|_v,|\sum_{i=1}^{l-1}m_i'\theta_i|_v,|\sum_{i=1}^{l}m_i\theta_i|_v\}^{d_v}\\
&\prod_{j=1}^{q+1}\max\{|\sigma_j\theta_1|,
|\sigma_j(\sum_{i=1}^{l-1}m_i'\theta_i)|,|\sigma_j(\sum_{i=1}^{l}m_i\theta_i)|\}^{d_j}.
\end{alignat*}
Because $\theta_1,...,\theta_l$ are in $\C_0^{-1}\D$
this is 
\begin{alignat*}3
\leq N(\C_0^{-1}\D)^{-1}
\prod_{j=1}^{q+1}(le)^{d_j}\max\{|\sigma_j\theta_1|,...,|\sigma_j\theta_l|\}^{d_j},
\end{alignat*}
and since $\prod_{j=1}^{q+1}\gamma_j^{d_j}=1$ this in
turn is
\begin{alignat*}3
&=(le)^d N(\C_0^{-1}\D)^{-1}
\prod_{j=1}^{q+1}\max\{\gamma_j|\sigma_j\theta_1|,.
..,\gamma_j|\sigma_j\theta_l|\}^{d_j}\\
&=(le)^d(C_{\en}^{fin})^d N(\D)^{-1}
\left(\prod_{j=1}^{q+1}\max\{\gamma_j|\sigma_j\theta_1|,.
..,\gamma_j|\sigma_j\theta_l|\}^{2d_j}\right)^{\frac{1}{2}}\\
&=(le)^d(C_{\en}^{fin})^d N(\D)^{-1}
\left(\prod_{j=1}^{q+1}|w_j|_{\infty}^{2d_j}\right)^{\frac{1}{2}}
\end{alignat*}
where $w_j$ is the vector $(\gamma_j\sigma_j\theta_1,.
..,\gamma_j\sigma_j\theta_l)$ in 
$\IR^{l}$ if $j\leq r$ and in $\IC^{l}$ if $j>r$ and $|\cdot|_{\infty}$ denotes the maximum norm.
Now using the inequality between the arithmetic and geometric mean and $|\cdot|\geq |\cdot|_{\infty}$
for the $l^2$-norm $|\cdot|$ we may estimate the above by
\begin{alignat}3
\nonumber&\leq (le)^d (C_{\en}^{fin})^d N(\D)^{-1}
\left(\frac{1}{d}\sum_{j=1}^{q+1}d_j|w_j|^{2}\right)^{\frac{d}{2}}\\
\label{aprip7}&\leq (le)^d (2/d)^{d/2}
(C_{\en}^{fin})^dN(\D)^{-1}
\left(\sum_{j=1}^{q+1}|w_j|^{2}\right)^{\frac{d}{2}}.
\end{alignat}
The vector $(\ta_0\sigma\theta_1,...,\ta_0\sigma\theta_l)$
in $\IR^{ld}$ has squared length exactly
\begin{alignat*}3
\sum_{j=1}^{q+1}|(\gamma_j\sigma_j\theta_1,.
..,\gamma_j\sigma_j\theta_l)|^{2}, 
\end{alignat*}
so that the right-hand side of (\ref{aprip7}) is
\begin{alignat}3
\label{K1}
&=(le)^d (2/d)^{d/2}
(C_{\en}^{fin})^d
N(\D)^{-1}|(\ta_0\sigma\theta_1,...,\ta_0\sigma\theta_l)|^d.
\end{alignat}
Moreover by (\ref{visuccmin}) one has
\begin{alignat}3
\label{K2}
|(\ta_0\sigma\theta_1,...,\ta_0\sigma\theta_l)|= 
(|v_1|^2+...+|v_l|^2)^{\frac{1}{2}}\leq \sqrt{l}\lambda_l.
\end{alignat}
Note that by definition $l\leq d$.
Combining (\ref{K1}) and (\ref{K2}) yields the 
desired result.
\end{proof}

\section{Application of counting}\label{subseccounting}
Recall the partition (\ref{partSF1}) of $S_F(T)$.
In this section we concentrate
on the component $S_{F({\bf 0})}(T)$. We will use Theorem \ref{TWi1} 
to estimate the number of points in 
$\G0\cap S_{F({\bf 0})}(T)$ satisfying a certain
primitivity condition.
Let $S_1\subseteq \sigma K^{n+1}$ and $S_2\subseteq \IR^D$ be
sets with $|S_1\cap S_2|$ or $|\tau S_1\cap S_2|$ finite. We use the following notation 
\begin{alignat}3
\label{defZstar}
Z^*(S_1,S_2)&=|\{\sigma\bom\in S_1\cap S_2;\bom\neq {\bf 0},
k(\omega_0:...:\omega_n)=K\}|\\
\label{defZstartau}
Z_{\tau}^{*}(\tau S_1,S_2)&=|\{\tau\sigma\bom\in \tau S_1\cap S_2;\bom\neq {\bf 0},
k(\omega_0:...:\omega_n)=K\}|.
\end{alignat}
We recall that $\tau$ and 
$\sigma$ are injective. Hence (\ref{defZstar}) and (\ref{defZstartau}) are well-defined and moreover
\begin{alignat}3
\label{ZstarZid1}
Z^*(S_1,S_2)=Z_{\tau}^{*}(\tau S_1,\tau S_2).
\end{alignat}
It might be worth to repeat (\ref{deflami1}) namely
\begin{alignat*}3
\lambda_i=\lambda_i(\g0)
\end{alignat*}
for $1\leq i \leq d$.\\
Recall also definition (\ref{mu2})
\begin{alignat*}3
\mu_g=m(e-g)(n+1)-1.
\end{alignat*}
Inclusion (\ref{SF0supset}) tells us in particular $S_{F({\bf 0})}(T)$ is bounded.\\

First suppose $q>0$.\\
We apply Lemma \ref{lemma2.8} not to $F$ but to  
\begin{alignat*}3
F(\v0)=[0,1)\frac{u_1}{n_1}+...+[0,1)\frac{u_q}{n_q}.
\end{alignat*}
Remember that by (\ref{defnj}) 
\begin{alignat*}3
|\frac{u_j}{n_j}|=\frac{|u_j|}{[|u_j|]+1}<1.
\end{alignat*}
We refer to (\ref{Flpb2}) and the observations just after
to conclude that $\partial F(\v0)$ lies in Lip$(q+1,2,2q,q-1)$. Furthermore
it is clear that $F(\v0)$ lies in a ball of radius 
$r_{F({\bf 0})}=q$. Applying Lemma \ref{lemma2.8} gives
that the boundary 
\begin{alignat}3
\label{boundinlip}
\partial S_{F({\bf 0})}(1)\text{ lies
in Lip}(D,1,\widetilde{M},\widetilde{L})
\end{alignat}
where
\begin{alignat*}3
\widetilde{M}&=(2q+1)\M^{q+1},\\
\widetilde{L}&=3\sqrt{D}(2q)\exp(\sqrt{q}(2q-1))(L+C_{\en}^{inf}).
\end{alignat*}
In the sequel it will sometimes be convenient to use Vinogradov's  $\ll$ notation. 
The implied constant will depend on $n$ and $d$ only. Thus we have
\begin{alignat*}3
\widetilde{M}&\ll \M^{q+1}\leq \M^d,\\
\widetilde{L}&\ll L+C_{\en}^{inf}.
\end{alignat*}

Now suppose $q=0$.\\ 
Therefore we have $S_{F({\bf 0})}(1)=S_{F}(1)$.
Recalling the observation just after Lemma \ref{lemma2.8}
shows directly that (\ref{boundinlip}) holds with $\widetilde{M}=\M \leq \M^d$ and $\widetilde{L}=L \leq L+C_{\en}^{inf}$.\\

By Theorem \ref{TWi1} we deduce
that $S_{F({\bf 0})}(1)$ is measurable.
Since by (\ref{homexp2}) $S_{F({\bf 0})}(T)=TS_{F({\bf 0})}(1)$  we conclude
that the latter remains true for $S_{F({\bf 0})}(T)$.
So the quantities 
$\Vol S_{F({\bf 0})}(T)$ and $|\G0\cap S_{F({\bf 0})}(T)|$ are well-defined and finite. 
\begin{proposition}\label{prop4}
With $A=A_{\en}$ as in Theorem \ref{prop2},
$T>0$ and $g=[K_0:k]$ as in Lemma \ref{lemma2.7} we have
\begin{alignat*}3
|Z_{\tau}^*(\G0,S_{F({\bf 0})}(T))-\frac{\Vol S_{F({\bf 0})}(T)}
{\det \G0}|\ll \frac{A T^{d(n+1)-1}}
{N\D^{n+1-1/d}\delta_{g}(K/k)^{\mu_g}}.
\end{alignat*}
\end{proposition}
\begin{proof}
Recall that $A=\M^{d}(C(L+1))^{d(n+1)-1}$.
We have 
\begin{alignat*}3
\mu_g=(d-mg)(n+1)-1\leq (d-l+1)(n+1)-1
\end{alignat*}
by Lemma \ref{lemma2.6} provided $l\geq 2$. But if 
$l=1$ then $K=k$ and thus $G(K/k)=\{1\}$, so $g=1$.
Hence for $l=1$ the inequality remains valid.
Thanks to Lemma \ref{lemma2.7} and (\ref{defc})
relating $C=C_{\en}$ and $C_{\en}^{inf}$ it is enough 
to verify the claim
\begin{alignat}3 
\label{claimprop4}
|Z_{\tau}^*(\G0,S_{F({\bf 0})}(T))-\frac{\Vol S_{F({\bf 0})}(T)}
{\det \G0}|\ll \M^{d}\frac{(C_{\en}^{inf}(L+1)T)^{d(n+1)-1}}
{\lambda_1^{(l-1)(n+1)}\lambda_l^{(d-l+1)(n+1)-1}}.
\end{alignat}
Remember also inclusion (\ref{SF0supset})
telling us
\begin{alignat}3
\label{OM2}
S_{F({\bf 0})}(T)\subseteq B_0(\d_1T)
\end{alignat}
where $\d_1=\sqrt{d(n+1)}C^{inf}_{\en}\exp(q)$.\\

We consider two cases.\\

$(1)$ \quad $T<\d_1^{-1}\lambda_l$.\\
\newline
Now (\ref{OM2}) shows that $|v|<\lambda_l$ for each 
$v$ in $S_{F({\bf 0})}(T)$. From (\ref{OG}) we get
$\G0 \subseteq \tau(\sigma(\C_0^{-1}\D)^{n+1})$ and so Lemma 
\ref{lemma2.5} implies 
\begin{alignat*}3
Z_{\tau}^*(\G0,S_{F({\bf 0})}(T))=0. 
\end{alignat*}
On the other hand
\begin{alignat*}3
\frac{\Vol S_{F({\bf 0})}(T)}{\det \G0}
&\leq 
\frac{\Vol B_0(\d_1T)}{\det \tau(\sigma(\C_0^{-1}\D)^{n+1})}.
\end{alignat*}
Since $\det(\Lambda_0^{n+1})=(\det\Lambda_0)^{n+1}$
for any lattice $\Lambda_0$ in $\IR^d$ the latter is
\begin{alignat*}3
= \frac{\Vol B_0(\d_1T)}{\det(\g0)^{n+1}}.
\end{alignat*}
Because of $\Vol B_0(R) \ll R^{d(n+1)}$, Minkowski's Second
Theorem and $(1)$ this in turn is
\begin{alignat*}3
&\ll \frac{(\d_1T)^{d(n+1)}}{\det(\g0)^{n+1}}
&&\ll\frac{(\d_1T)^{d(n+1)}}
{(\lambda_1...\lambda_d)^{n+1}}\\
&\ll\frac{\lambda_l(\d_1T)^{d(n+1)-1}}
{(\lambda_1...\lambda_d)^{n+1}}
&&\ll \frac{(C^{inf}_{\en}T)^{d(n+1)-1}}
{\lambda_1^{(l-1)(n+1)}\lambda_l^{(d-l+1)(n+1)-1}}.
\end{alignat*}
This implies (\ref{claimprop4}) in case $(1)$ because
$\M\geq 1$.\\

$(2)$ \quad $T\geq \d_1^{-1}\lambda_l$.\\
\newline
Thus for $1 \leq i \leq l$ one has 
\begin{alignat}3
\label{(2')}
C^{inf}_{\en}\frac{T}{\lambda_i}\gg 1.
\end{alignat}
Set
\begin{alignat*}3
&S=\G0\cap S_{F({\bf 0})}(T).
\end{alignat*}
Notice that by definition (\ref{inkl1}) $\v0$ is not in $S_{F({\bf 0})}(T)$ for all $T>0$. Thus we can define 
\begin{alignat*}3
&S'=\{v\in S; v=(\ta_0\sigma\omega_0,...,\ta_0\sigma\omega_n),
k(\omega_0:...:\omega_n)\subsetneq K\}.
\end{alignat*}
Clearly
\begin{alignat*}3
Z_{\tau}^*(\G0,S_{F({\bf 0})}(T))=|S|-|S'|.
\end{alignat*}
Let us estimate $|S|$ first.
Due to (\ref{boundinlip}) we know that
$\partial S_{F({\bf 0})}(1)$ lies in Lip$(D,1,\widetilde{\M},\widetilde{L})$
where $\widetilde{\M}\ll M^d$ and $\widetilde{L}\ll L+C_{\en}^{inf}$.
By (\ref{homexp2}) we see that
$\partial S_{F({\bf 0})}(T)$ is in Lip$(D,1,\widetilde{\M},\widetilde{L}T)$.
Next we apply Theorem \ref{TWi1} 
of Section \ref{1subsec3} to deduce
\begin{alignat}3
\nonumber||S|-\frac{\Vol S_{F({\bf 0})}(T)}{\det \G0}|
&\ll \widetilde{\M} \max_{0\leq j\leq d(n+1)-1}
\frac{(\widetilde{L}T)^j}{\lambda_1(\G0)...\lambda_j(\G0)}\\
\label{M1Test1}
&\ll \M^{d}\max_{0\leq j\leq d(n+1)-1}
\frac{((L+C_{\en}^{inf})T)^j}{\lambda_1(\G0)...\lambda_j(\G0)}.
\end{alignat}
From (\ref{OG}) we get
\begin{alignat}3
\label{minbou1}
\lambda_j(\G0)\geq \lambda_j((\g0)^{n+1})
\end{alignat}
for $1\leq j \leq d(n+1)$.
We abbreviate the right-hand side of (\ref{minbou1})
to $\abbmin_j$.
Inserting this estimate in (\ref{M1Test1}) and then using $C_{\en}^{inf}\geq 1$ in the form $L+C_{\en}^{inf}\leq (L+1)C_{\en}^{inf}$ 
yields the bound 
\begin{alignat}3
\label{M1Test2}
&\ll \M^{d}(L+1)^{d(n+1)-1}\max_{0\leq j\leq d(n+1)-1}
\frac{(C_{\en}^{inf}T)^j}
{\abbmin_1...\abbmin_j}.
\end{alignat}
Consider the expressions
\begin{alignat}3
\label{Ej}
E_j=\frac{(C_{\en}^{inf}T)^j}{\abbmin_1...\abbmin_j}
\end{alignat}
in (\ref{M1Test2}).
From Lemma \ref{minpowlatt} we see that 
$\abbmin_1,...,\abbmin_D$ are 
\begin{alignat*}3
\lambda_1,...,\lambda_1,\lambda_2,...,\lambda_2,...,\lambda_d,...,\lambda_d
\end{alignat*}
in blocks of $n+1$. Thus for $j\leq (l-1)(n+1)$ we have
$\abbmin_j\leq \lambda_l$. So in this case 
\begin{alignat}3
\label{Ejineq1}
E_j=E_{j-1}\frac{C_{\en}^{inf}T}{\abbmin_j}\gg E_{j-1}.
\end{alignat}
Therefore the maximum over these $j$ in (\ref{Ej}) is
\begin{alignat}3
\label{Ejineq2}
\ll E_{(l-1)(n+1)}=\frac{(C_{\en}^{inf}T)^{(l-1)(n+1)}}
{(\lambda_1...\lambda_{l-1})^{n+1}}\leq 
\frac{(C_{\en}^{inf}T)^{(l-1)(n+1)}}
{\lambda_1^{(l-1)(n+1)}}.
\end{alignat}
For the other $j>(l-1)(n+1)$ we get $\abbmin_j\geq \lambda_l$ so
\begin{alignat}3
\label{Ejineq3}
E_j\leq E_{j-1}\frac{C_{\en}^{inf}T}{\lambda_l}
\end{alignat}
which contribute an extra 
\begin{alignat*}3
\left(\frac{C_{\en}^{inf}T}{\lambda_l}\right) ^{d(n+1)-1-(l-1)(n+1)}\gg 1
\end{alignat*}
to the maximum in (\ref{Ejineq2}). This yields
the bound 
\begin{alignat}3
\label{M1Test3}
\ll \M^{d}(C^{inf}_{\en}(L+1))^{d(n+1)-1}\frac{T^{d(n+1)-1}}{\lambda_1^{(l-1)(n+1)}\lambda_l^{(d-l+1)(n+1)-1}}
\end{alignat}
for (\ref{M1Test2}).\\

Next we shall obtain an upper bound for $|S'|$.
For $(\ta_0\sigma\omega_0,...,\ta_0\sigma\omega_n)$ in $S'$
the field $k(\omega_0:...:\omega_n)$ lies in a strict subfield, say $K_1$, of $K$.
Hence there exist two different embeddings 
$\sigma_a, \sigma_b$ of $K$ with 
\begin{alignat*}3
\sigma_a\alpha=\sigma_b\alpha
\end{alignat*}
for all $\alpha$ in $K_1$.
Now $(\ta_0\sigma\omega_0,...,\ta_0\sigma\omega_n)\neq \v0$
hence at least one of the numbers $\omega_0,...,\omega_n$ is non-zero.
By symmetry we lose only a factor $n+1$ if we assume 
$\omega_0 \neq 0$. So let us temporarily regard $\omega_0 \neq 0$ as fixed;
then every $\omega_j$ for $1\leq j \leq n$ satisfies
\begin{alignat*}3
\sigma_a\frac{\omega_j}{\omega_0}=\sigma_b\frac{\omega_j}{\omega_0}.
\end{alignat*}
Let $z_0,z_1$ be in $\IR$ with $z_0+iz_1=\frac{\sigma_a\omega_0}{\sigma_b\omega_0}$.
Then we get
\begin{alignat*}3
\Re\sigma_a\omega_j&=
z_0\Re\sigma_b\omega_j-z_1\Im\sigma_b\omega_j,\\
\Im \sigma_a\omega_j&=z_1\Re\sigma_b\omega_j+
z_0\Im\sigma_b\omega_j,
\end{alignat*}
where we used $\Re$ for the real and $\Im$ for the
imaginary part of a complex number.
This shows that all $\sigma \omega_j$ for $1\leq j\leq n$ lie 
in a hyperplane $\P(\omega_0)$ of $\IR^d$ and therefore all $\ta_0\sigma \omega_j$ lie in the hyperplane $\ta_0\P(\omega_0)$. 
The inclusion (\ref{OM2}) implies 
$|\ta_0\sigma\omega_j|\leq \d_1T$.
The intersection of a ball with radius $r$ 
and a hyperplane in $\IR^d$ is a ball in some
$\IR^{d-1}$ with radius $r'\leq r$. It is easy to
see that it belongs to the class Lip$(d,1,1,2\sqrt{d-1}r)$
(for example using (\ref{LP1}) from Appendix with $q=d-1$ and $r_F=\sqrt{d-1}r'$ if the center is at the origin).
Moreover its $d$-dimensional volume is zero. Hence by
Theorem \ref{TWi1} and (\ref{(2')}) we obtain the upper
bound
\begin{alignat*}3
\ll \max_{0\leq i<d}\frac{(\d_1T)^i}{\lambda_1...\lambda_i}
\ll \frac{(C^{inf}_{\en}T)^{d-1}}
{\lambda_1^{l-1}\lambda_l^{d-l}}
\end{alignat*}
for the number of $\ta_0\sigma\omega_j$ 
with $1\leq j\leq n$.\\

Next we have to estimate the number of $\ta_0\sigma\omega_0$. By inclusion (\ref{OM2}) we see once more that $|\ta_0\sigma\omega_0|\leq \d_1T$.
Now by virtue of Theorem \ref{TWi1} we
deduce the following upper bound
\begin{alignat*}3
\ll \frac{\Vol B_0(\d_1T)}{\det \g0}+
\max_{0\leq i < d}\frac{(\d_1T)^i}{\lambda_1...\lambda_i}
\end{alignat*}
for the number of  $\ta_0\sigma\omega_0$.
Going right up to the last minimum, we see that this is bounded by
\begin{alignat*}3
\ll \max_{0\leq i\leq d}\frac{(\d_1T)^i}{\lambda_1...\lambda_i}
\end{alignat*}
and taking (\ref{(2')}) into account yields
the upper bound
\begin{alignat*}3
\ll \frac{(C^{inf}_{\en}T)^d}{\lambda_1^{l-1}\lambda_l^{d-l+1}}.
\end{alignat*}
Multiplying the bounds for the number of $\ta_0\sigma\omega_0$ and 
$\ta_0\sigma\omega_j$ and then summing over all strict subfields $K_1$ of $K$ leads to  
\begin{alignat*}3
|S'|\ll \frac{(C^{inf}_{\en}T)^d}{\lambda_1^{l-1}\lambda_l^{d-l+1}} \left(\frac{(C^{inf}_{\en}T)^{d-1}}{\lambda_1^{l-1}\lambda_l^{d-l}}\right)^{n}=\frac{(C^{inf}_{\en}T)^{d(n+1)-n}}
{\lambda_1^{(l-1)(n+1)}\lambda_l^{(d-l+1)(n+1)-n}}.
\end{alignat*}
We appeal once more to (\ref{(2')}) with $i=l$ to see that 
the latter is
\begin{alignat*}3
\ll \frac{(C^{inf}_{\en}T)^{d(n+1)-1}}
{\lambda_1^{(l-1)(n+1)}\lambda_l^{(d-l+1)(n+1)-1}}.
\end{alignat*}
Combining the estimates for $|S|$ and $|S'|$
proves the claim (\ref{claimprop4}) in case $(2)$, hence the proposition.
\end{proof}

\section{Proof of Theorem 3.1}\label{2endofproof}
Let $\Lamst(\A)$ be the subset of
$\Lam(\A)$ defined by
\begin{alignat*}3
\Lamst(\A)=\{\sigma(\balf); \balf \in K^{n+1}, 
N_v(\sigma_v\balf)=|\A|_v \text{ for all finite }v \}.
\end{alignat*}
Recall also definition (\ref{defZstar}).
As in Section \ref{subseccounting} the star $^*$ indicates
some primitivity condition. However, the 
property defining the set above has nothing to do with
the one in Section \ref{subseccounting}.

\begin{lemma}\label{ZKHX}
For $X>0$ we have
\begin{alignat*}3
Z_{\en}(\IP^n(K/k),X)=
w_K^{-1}\sum_{\A\in R}Z^*(\Lamst(\A),S_{F}(N\A^{\frac{1}{d}}X))
\end{alignat*}
where the sum runs over any system $R$ of ideal class representatives
of $K$.
\end{lemma}
\begin{proof}
Let $P\in \IP^n(K)$ with homogeneous coordinates
$(\alpha_0,...,\alpha_n)=\balf \in K^{n+1}\backslash \{\bf{0}\}$. 
Thanks to the uniqueness of the prime factorization
for non-zero fractional ideals together with property
$N_v(\sigma_v K^{n+1})\subseteq \Gamma_v$, we may conclude
that there is exactly one ideal $\A=\A_{\balf}$ such 
that 
\begin{alignat}3
\label{NvId1}
N_v(\sigma_v\balf)=|\A|_v  
\end{alignat}
for all finite $v$.
Suppose $\varepsilon \in K^*$ then we have
\begin{alignat*}3
N_v(\sigma_v\varepsilon\balf)=|\sigma_v\varepsilon|_v
N_v(\sigma_v\balf)
\end{alignat*}
for all finite $v$.
Hence $\A_{\varepsilon\balf}=\varepsilon\A_{\balf}$; 
in other
words the ideal class of $\A_{\balf}$ is independent
of the coordinates $\balf$ we have chosen.
In particular we can choose $\balf$ such that
$\A_{\balf}$ lies in $R$ and so
$\balf$ is unique up to units $\eta$.
The set $F(\infty)=F+\IR \vdelta$ is
a fundamental set of $\IR^{q+1}$ under the 
action of the additive subgroup $l(\IU)$.
Because of $(ii)$ of 
Section \ref{secALH}
we have  
\begin{alignat*}3
\log N_i(\sigma_i(\eta\balf))^{d_i}
=\log N_i(\sigma_i\balf)^{d_i}+d_i\log|\sigma_i\eta|
\end{alignat*}
for $1\leq i \leq q+1$.
And so there exist exactly $\wK$ representatives
$\balf$ of $P$ with 
\begin{alignat*}3
(d_1\log N_1(\sigma_1\balf),...,d_{q+1}\log N_{q+1}(\sigma_{q+1}\balf))
\in F(\infty).
\end{alignat*}
But the above is equivalent with  
\begin{alignat*}3
(N_1(\sigma_1\balf)^{d_{1}},...,N_{q+1}(\sigma_{q+1}\balf)^{d_{q+1}})
\in \exp(F(\infty)).
\end{alignat*}
Furthermore
\begin{alignat*}3
\exp(F(T_0))=\{(X_1,...,X_{q+1})\in \exp(F(\infty)); 
X_1...X_{q+1}\leq T_0^d\}.
\end{alignat*}
By definition (see end of Section \ref{secALH}) $\hen^{inf}(\balf), \hen^{fin}(\balf)$  are invariant under substitution of $\balf$
by $\omega\balf$ where $\omega$ denotes a root of unity in $K$.
Hence for all $\wK$ possible choices $\balf$ of $P$
the inequality
\begin{alignat*}3
\hen^{inf}(\balf)\leq T_0 
\end{alignat*}
is equivalent to
\begin{alignat*}3
\sigma\balf \in S_F(T_0).
\end{alignat*}
On the other hand 
\begin{alignat*}3
\hen(P)=\hen^{inf}(\balf)\hen^{fin}(\balf)
\end{alignat*}
and by (\ref{NvId1})
\begin{alignat*}3
\hen^{fin}(\balf)^d=\prod_{v\nmid \infty}|\A|_v^{d_v}=N\A^{-1},
\end{alignat*}
which completes the proof.
\end{proof}
 
Let $\Cl$ be the set of ideal classes
and for (non-zero) ideals $\A$, $\B$, $\C$ 
denote by $\mathcal{A}$, $\mathcal{B}$, $\mathcal{C}$ the ideal classes of $\A$,
$\B$ and $\C$.
Recall from (\ref{deltawelldef1}) that the
function $\Delta_{\en}(\cdot)$ is well-defined on $\Cl$.
\begin{lemma}\label{lemmaVfineq2}
We have
\begin{alignat}3
\label{Vfineq2}
\sum_{\A \in R}\sum_{\B}\frac{\mu(\B)}{N\B^{n+1}}
\Delta_{\en}(\mathcal{A}\mathcal{B})^{-1}
=\frac{1}{\zeta_K(n+1)}\sum_{\mathcal{D} \in \Cl}\Delta_{\en}(\mathcal{D})^{-1}
\end{alignat}
where the inner sum on the left-hand side runs over 
all non-zero ideals $\B$ in $\Oseen_K$.
\end{lemma}
\begin{proof}
We have 
\begin{alignat*}3
\sum_{\A \in R}\sum_{\B}\frac{\mu(\B)}{N\B^{n+1}}
\Delta_{\en}(\mathcal{A}\mathcal{B})^{-1}=
&\sum_{\mathcal{A} \in \Cl}\sum_{\B}\frac{\mu(\B)}{N\B^{n+1}}
\Delta_{\en}(\mathcal{A}\mathcal{B})^{-1}\\
=&\sum_{\mathcal{A} \in \Cl}\sum_{\mathcal{C} \in \Cl}
\Delta_{\en}(\mathcal{A}\mathcal{C})^{-1}
\sum_{\B \in \mathcal{C} }\frac{\mu(\B)}{N\B^{n+1}}\\
=&\sum_{\mathcal{A} \in \Cl}\sum_{\mathcal{D} \in \Cl}
\Delta_{\en}(\mathcal{D})^{-1}
\sum_{\B \in \mathcal{D}/\mathcal{A} }\frac{\mu(\B)}{N\B^{n+1}}\\
=&\sum_{\mathcal{D} \in \Cl}
\Delta_{\en}(\mathcal{D})^{-1}
\sum_{\mathcal{A} \in \Cl} \sum_{\B \in \mathcal{D}/\mathcal{A} }\frac{\mu(\B)}{N\B^{n+1}}\\
=&\sum_{\mathcal{D} \in \Cl}
\Delta_{\en}(\mathcal{D})^{-1}
\sum_{\B}\frac{\mu(\B)}{N\B^{n+1}}
\end{alignat*}
where the last sum is over all non-zero ideals $\B$ in $\Oseen_K$.
Now we just have to remember the fact that 
$\sum_{\B}\frac{\mu(\B)}{N\B^{s}}=\zeta_K(s)^{-1}$
for $s>1$ (so in particular for $s=n+1$)
and the result drops out.
\end{proof}

The image of $\sigma_v(K^{n+1}\backslash\{{\bf 0}\})$ 
under the map $N_v$ lies in $\Gamma_v^*$ and for 
all non-zero $\balf$ in $K^{n+1}$ there are only
finitely many $v$ with $N_v(\sigma_v\balf)\neq 1$.
So assume $\balf$ is in $K^{n+1}\backslash\{{\bf 0}\}$;
then $N_v(\sigma_v\balf)\leq |\A|_v$ for all $v\nmid \infty$ 
is equivalent with the existence of a unique $\B=\B(\balf) \subseteq \Oseen_K$, $\B \neq 0$ such that 
$N_v(\sigma_v\balf)=|\A\B|_v$ for all $v \nmid \infty$.
Hence from (\ref{defLamen}) we have the following disjoint union
\begin{alignat*}3
\Lam(\A)=\bigcup_{\B}\Lamst(\A\B)
\end{alignat*}
and therefore
\begin{alignat*}3
Z^*(\Lam(\A),S_F(T))=
\sum_{\B}Z^*(\Lamst(\A\B),S_F(T))
\end{alignat*}
for any $T>0$.
Using the M\"obius function $\mu_K$ of $K$ we get by inversion
\begin{alignat}3
\label{Moebinv1}
Z^*(\Lamst(\A),S_F(T))=
\sum_{\B}\mu_K(\B)Z^*(\Lam(\A\B),S_F(T)).
\end{alignat}
Applying (\ref{partSF1}) we find 
\begin{alignat}3
\nonumber Z^*(\Lam(\A\B),S_F(T))
&=\sum_{\i}Z^*(\Lam(\A\B),S_{F({\bf i})}(T))
\end{alignat}
where $\i$ is taken over the same set as in (\ref{partSF1}).
Referring to (\ref{ZstarZid1}) we see that the latter is 
\begin{alignat}3
\nonumber &=\sum_{\i}Z_{\ti}^*(\ti\Lam(\A\B),\ti S_{F({\bf i})}(T))
\end{alignat}
and by (\ref{tr1}) this in turn is 
\begin{alignat}3
\nonumber &=\sum_{\i}Z_{\ti}^*(\ti\Lam(\A\B),S_{F({\bf 0})}(T)).
\end{alignat}
Thus 
\begin{alignat}3
\label{ZZtaui}
Z^*(\Lam(\A\B),S_F(T))
=\sum_{\i}Z_{\ti}^*(\ti\Lam(\A\B),S_{F({\bf 0})}(T))
\end{alignat}
and again $\i$ is taken over the same set as in (\ref{partSF1}).
Next we apply Proposition \ref{prop4}
with $\D=\A\B$. To emphasize the dependence on $\i$ and $\A\B$ we can think of $g=g(\i,\A\B)$. We get
\begin{alignat*}3
Z_{\ti}^*(\ti\Lam(\A\B),S_{F({\bf 0})}(T))=
\frac{\Vol S_{F({\bf 0})}(T)}{\det \ti\Lam(\A\B)}
 +O\left(\frac{A T^{d(n+1)-1}}
{(N\A\B)^{n+1-1/d}\delta_{g}(K/k)^{\mu_{g}}}\right).
\end{alignat*}
By (\ref{dettaui}) we have
$\det \ti\Lam(\A\B)=\det \Lam(\A\B)$
and taking also into account (\ref{tr1}) and (\ref{partSF1}) gives
\begin{alignat*}3
\sum_{\i}\Vol S_{F({\bf 0})}(T)=\sum_{\i}\Vol \ti S_{F(\i)}(T)=\sum_{\i}\Vol S_{F(\i)}(T)=\Vol S_{F}(T).
\end{alignat*}
Referring back to (\ref{ZZtaui}) we conclude
\begin{alignat}3
\label{Absch1}
\nonumber Z^*(\Lam(\A\B),S_F(T))&=\sum_{\i}Z_{\ti}^*(\ti\Lam(\A\B),S_{F({\bf 0})}(T))\\
&=\frac{\Vol S_F(T)}{\det \Lam(\A\B)}
 +O\left(\frac{A T^{d(n+1)-1}}
{(N\A\B)^{n+1-1/d}}\sum_{\i}\delta_{g}(K/k)^{-\mu_{g}}\right).
\end{alignat}
Let us focus on the error term.
Recall that $g=g(\i,\A\B)=[K_0:k] \in G=G(K/k)$ where $K_0=k(\theta_1/\theta_1,...,\theta_{l-1}/\theta_1)$
if $l\geq 2$ and $K_0=k$ if $l=1$.
Thus the $\sum_{\i}$ above can be replaced by $t\sum_{g\in G}$
with $t=\sum_{\i}1$.
By (\ref{tAbsch}) we have 
\begin{alignat*}3
t&\ll R_K
\end{alignat*}
and (\ref{homexp1}) says
\begin{alignat*}3
S_F(T)&=TS_F(1).
\end{alignat*}
Thus by (\ref{Moebinv1}) we get
\begin{alignat}3
\label{Moebinv2}
Z^*(\Lamst(\A),S_F(T))=&\sum_{\B}\mu_K(\B)\frac{\Vol S_F(1)T^{d(n+1)}}{\det \Lam(\A\B)}\\
\nonumber+&O\left(\sum_{\B}\frac{AR_KT^{d(n+1)-1}}
{(N\A\B)^{n+1-1/d}}\sum_{g\in G}\delta_{g}(K/k)^{-\mu_{g}}\right).
\end{alignat}
According to Lemma \ref{ZKHX} we set 
\begin{alignat*}3
T=T(\A)=N\A^{\frac{1}{d}}X.
\end{alignat*}
By (\ref{idkl}) we see that 
\begin{alignat*}3
\det \Lam(\A\B)=\Delta_{\en}(\mathcal{A}\mathcal{B})(N\A\B)^{n+1}
\end{alignat*}
for the corresponding ideal classes $\mathcal{A}, \mathcal{B}$.
Therefore (\ref{Moebinv2}) with $T=N\A^{\frac{1}{d}}X$ is equal 
\begin{alignat*}3
&\sum_{\B}\frac{\mu_K(\B)}{N\B^{n+1}}\Delta_{\en}(\mathcal{A}\mathcal{B})^{-1}\Vol S_F(1)X^{d(n+1)}\\
+
&O\left(\sum_{\B}\frac{AR_KX^{d(n+1)-1}}
{N\B^{n+1-1/d}}\sum_{g\in G}\delta_{g}(K/k)^{-\mu_{g}}\right).
\end{alignat*}
Lemma \ref{ZKHX} tells us that this quantity has to
be summed over a set $R$ of ideal class representatives 
$\A$ and divided by the number $\wK$ of roots of unity. Applying Lemma \ref{lemmaVfineq2} yields
\begin{alignat*}3
Z_{\en}(\IP^n(K/k),X)&=\frac{1}{\zeta_K(n+1)\wK}
\sum_{\mathcal{D}\in \Cl}\Delta_{\en}(\mathcal{D})^{-1}\Vol S_F(1)X^{d(n+1)}\\
&+O\left(\sum_{\B}\frac{Ah_KR_KX^{d(n+1)-1}}
{N\B^{n+1-1/d}}\sum_{g\in G}\delta_{g}(K/k)^{-\mu_{g}}\right).
\end{alignat*}
By (\ref{defVfin1}) we have 
\begin{alignat*}3
\sum_{\mathcal{D} \in \Cl}\Delta_{\en}(\mathcal{D})^{-1}=
2^{s_K(n+1)}h_K V_{\en}^{fin}|\Delta_K|^{-\frac{n+1}{2}}.
\end{alignat*}
The volume of $S_F(1)$ has been computed
by Masser and Vaaler in \cite{1} Lemma 4 
\begin{alignat*}3
\Vol S_F(1)=(n+1)^q R_K V_{\en}^{inf}.
\end{alignat*}
On recalling that $V_{\en}=V_{\en}^{fin}V_{\en}^{inf}$
we end up with
\begin{alignat*}3
&\frac{1}{\zeta_K(n+1)\wK}2^{s_K(n+1)}h_K V_{\en}^{fin}|\Delta_K|^{-\frac{n+1}{2}}(n+1)^q R_K V_{\en}^{inf}X^{d(n+1)}\\
=&S_K(n)2^{-r_K(n+1)}\pi^{-s_K(n+1)}V_{\en}X^{d(n+1)}
\end{alignat*}
for the main term - exactly the main term of the theorem.\\

To deal with the error term we assume first 
$(n,d)\neq(1,1)$. 
It is well-known that $\zeta_{K}(x)\leq \zeta_{\IQ}(x)^d$ for $x>1$
(see Lang \cite{13} p.322).
Thus we have
\begin{alignat*}3
\sum_{\B}N\B^{-(n+1-1/d)}\ll 1
\end{alignat*}
and so we are done. 
Next assume $(n,d)=(1,1)$ so $k=K=\IQ$, $q=0$
and therefore
$S_{F({\bf 0})}(T)=S_{F}(T)$.
By (\ref{SF0supset}) we have $S_{F}(T)\subseteq B_0(\kappa T)$
and here $\kappa =\sqrt{2}C_{\en}^{inf}$.
From Lemma \ref{lemma2.7} we get $\lambda_1\geq (1/\sqrt{2})(C_{\en}^{fin})^{-1}N\D$. It follows without difficulty that $B_0(\kappa T)$
contains no point of the lattice $(\sigma \C_0^{-1}\D)^2$
except the origin provided $T<(1/2)C_{\en}^{-1}N\D$.
But the origin does not lie in $S_F(T)$ and on recalling
the inclusion (\ref{OG}) we deduce $S_F(T)\cap \Lamen(\D)$
is empty for $T<(1/2)C_{\en}^{-1}N\D$. Hence we may restrict
the sum over $\B$ in (\ref{Moebinv1}) to 
$N\B\leq 2C_{\en}TN\A^{-1}$. Thus by
(\ref{Moebinv1})
\begin{alignat*}3
Z^*(\Lamst(\A),S_F(T))=
\sum_{\B \atop N\B\leq 2C_{\en}TN\A^{-1}}\mu_K(\B)Z^*(\Lam(\A\B),S_F(T))
\end{alignat*}
and by (\ref{Moebinv2}) we get for the latter  
\begin{alignat*}3
\sum_{\B \atop N\B\leq 2C_{\en}TN\A^{-1}}\mu_K(\B)\frac{\Vol S_F(1)T^{2}}{\det \Lam(\A\B)}+
O\left(\sum_{\B \atop N\B\leq 2C_{\en}TN\A^{-1}}\frac{AR_KT}
{N\A\B}\sum_{g\in G}\delta_{g}(K/k)^{-\mu_{g}}\right).
\end{alignat*}
Here $G=\{1\}$ and $\delta_{g}=1$.
Now in order to get the main term as in the 
case $(n,d)\neq (1,1)$ we let the sum run over all non-zero $\B$
in $\Oseen_K$ and correct by an additional error term
\begin{alignat*}3
\sum_{\B}\mu_K(\B)\frac{\Vol S_F(1)T^{2}}{\det \Lam(\A\B)}
+&O\left(\sum_{\B \atop N\B>2C_{\en}TN\A^{-1}}\frac{\Vol S_F(1)T^{2}}{\det \Lam(\A\B)}\right)\\
+&O\left(\sum_{\B \atop N\B\leq 2C_{\en}TN\A^{-1}}\frac{AR_KT}
{N\A\B}\right).
\end{alignat*}
We set $T=XN\A$ and by Lemma \ref{ZKHX} we see that this 
quantity has to
be summed over a set $R$ of ideal class representatives 
$\A$ and divided by the number $\wK$ of roots of unity.
But here $K=\IQ$ so $R$ consists just of a single class, $\wK=2$ and $R_K=1$. 
Thus
\begin{alignat*}3
Z_{\en}(\IP^n(K/k),X)=&2^{-1}\sum_{\B}\mu_K(\B)\frac{\Vol S_F(1)(XN\A)^{2}}{\det \Lam(\A\B)}\\
+&O\left(\sum_{\B \atop N\B>2C_{\en}X}\frac{\Vol S_F(1)(XN\A)^{2}}{\det \Lam(\A\B)}\right)
+O\left(\sum_{\B \atop N\B\leq 2C_{\en}X}\frac{AX}
{N\B}\right).
\end{alignat*}
As in the previous case the first term leads exactly to
the predicted main term. 
For the first error term we appeal once more to
(\ref{SF0supset}) to get 
$\Vol S_F(1)\ll(C_{\en}^{inf})^2$.
Using inclusion (\ref{OG}) we get
$\Lamen(\A\B)\subseteq (\sigma\C_0^{-1}\A\B)^2$
and therefore 
\begin{alignat*}3
\det \Lamen(\A\B)\geq \det (\sigma\C_0^{-1}\A\B)^2=(C_{\en}^{fin})^{-2}(N\A N\B)^2.
\end{alignat*}
So the first error term is reduced to
\begin{alignat*}3
C_{\en}^2 X^2\sum_{\B \atop N\B>2C_{\en}X}N\B^{-2}
\end{alignat*}
and so is  
\begin{alignat*}3
O(C_{\en}X)=O(AX\L).
\end{alignat*}
The second error term is even easier; namely
\begin{alignat*}3
\sum_{\B \atop N\B\leq 2C_{\en}X}\frac{AX}
{N\B}\leq AX\max\{0,1+\log(2C_{\en}X)\}=O(AX\L).
\end{alignat*}
This completes the proof of Theorem \ref{prop3}.

\appendix
\section{Proof of Lemma 7.1.}
Using the notation of Section 6 let us first recall the statement of the lemma.
\begin{lemma}
Suppose $q\geq 1$ and let $F$ be a set in $\Sigma$ such that
$\partial F$ is in Lip$(q+1,2,\M',L')$ and moreover assume 
$F$ lies in $B_0(r_F)$. Then 
$\partial S_{F}(1)$ is in Lip$(D,1,\widetilde{\M},\widetilde{L})$
where one can choose
\begin{alignat*}3
\widetilde{\M}&=(\M'+1)\M^{q+1}\\
\widetilde{L}&=3\sqrt{D}(L'+r_F+1)
\exp(\sqrt{q}(L'+r_F))(L+C_{\en}^{inf}).
\end{alignat*}
\end{lemma}
\begin{proof}
For $1\leq i\leq \M'$ let
\begin{alignat*}3
\psi^{(i)}:[0,1]^{q-1}\longrightarrow \IR^{q+1}
\end{alignat*}
be the parameterizing maps of $\partial F$ with Lipschitz
constants $L'$. 
Choose an orthonormal basis $e_1,...,e_q$ of $\Sigma$.
The affine map $\nu:[0,1]^q\longrightarrow \Sigma$ defined by
\begin{alignat}3 
\label{LP1}
\nu({\bf t})=(1-2t_1)r_F e_1+...+(1-2t_q)r_F e_q
\end{alignat}
is a  Lipschitz parameterization covering the topological closure
$\overline{F}$ with Lipschitz constant $2r_F$.
Since $\vdelta$ is not in $\Sigma$ the boundary
$\partial F(1)$ consists of two parts
\begin{alignat*}3
\partial(F(1))=(\partial(F)+(-\infty,0]\vdelta)\cup \overline{F}.
\end{alignat*}
So we see that $\partial(F(1))$ is parameterized by $\M'+1$ maps.
Here the parameter domain is not compact anymore but
this problem can easily be eliminated as we shall see in a moment.
Since $F$ is bounded we may use (\ref{vecsum1}) to get
\begin{alignat}3
\nonumber
\partial(\exp(F(1)))&=\exp(\partial(F(1)))\cup\{{\bf 0}\}\\
\label{paexF1}&=\exp(\partial(F)+(-\infty,0]\vdelta)\cup \exp(\overline{F})\cup \{{\bf 0}\}.
\end{alignat}
With a $\psi=(\psi_1,...,\psi_{q+1})=\psi^{(i)}$
as above, the first part is covered by
\begin{alignat}3
\label{1caseFI}
\Phi=\exp(\psi+t\vdelta)
=(e^{\psi_1+td_1},...,e^{\psi_{q+1}+td_{q+1}})
=(e^{\psi_1}u^{d_1},...,e^{\psi_{q+1}}u^{d_{q+1}})
\end{alignat}
with parameter domain $[0,1]^{q-1}\times (-\infty,0]$ and
$u=e^t$ in $(0,1]$.
Now we simply choose $u$ as parameter instead of $t$ and
extend its parameter range from $(0,1]$ to $[0,1]$ to cover
the origin. The remaining part of (\ref{paexF1}) is covered by
\begin{alignat}3
\label{2caseFI}
\Phi=\exp(\nu).
\end{alignat}
We use ${\bf t}$ for the parameter variables in
$[0,1]^{q}$, not just for (\ref{2caseFI}) as in
(\ref{LP1}) but also for (\ref{1caseFI}). 
So until now we have $\M'+1$ maps. We denote them
by $\Phi^{(i)}$ for $1\leq i \leq \M'+1$ or more simply 
$\Phi$. The $N_i$ are continuous functions and therefore
$\partial S_F(1)$ consists of these $({\vz}_1,...,{\vz}_{q+1})$ 
in $\prod_{i=1}^{q+1}\IR^{d_i(n+1)}=\IR^{d(n+1)}$ such that
\begin{alignat*}3
(N_1({\vz}_1)^{d_1},...,N_{q+1}({\vz}_{q+1})^{d_{q+1}}) \in \partial(\exp(F(1))).
\end{alignat*}
By our assumptions on $\en$ there are maps
\begin{alignat}3
\label{etai1}
\eta_i^{(j)}:[0,1]^{d_i(n+1)-1}\longrightarrow \IR^{d_i(n+1)}
\end{alignat}
for $1\leq i \leq q+1$ and $1\leq j \leq \M$ satisfying
a Lipschitz condition and whose
images cover the sets
\begin{alignat}3
\label{bouB1}
\{{\vz}\in \IR^{d_i(n+1)};N_i({\vz})=1\}.
\end{alignat}
We write more simply $\eta_i$.
For real $\zeta\geq 0$ the images of $\zeta\eta_i$
cover the sets $\{{\vz}\in \IR^{d_i(n+1)};N_i({\vz})=\zeta\}$ and with $\Phi=(\Phi_1,...,\Phi_{q+1})$ we obtain a parameterization of $\partial S_F(1)$ by maps
\begin{alignat}3
\label{paparSF}
(\Phi_1({\bf t})^{\frac{1}{d_1}}\eta_1({\bf \underline{t}^{(1)}}),...,
\Phi_{q+1}({\bf t})^{\frac{1}{d_{q+1}}}\eta_{q+1}
({\bf \underline{t}^{(q+1)}})).
\end{alignat}
We have  $\M'+1$ possibilities for $\Phi$
and $\M$ possibilities for each $\eta_i$. Hence the total
number of parameterization maps is $(\M'+1)\M^{q+1}$ and
the number of parameters is $q+\sum_{i=1}^{q+1}(d_i(n+1)-1)=d(n+1)-1=D-1$
as desired.\\

To verify the Lipschitz conditions and to compute 
a Lipschitz constant we make use of the following
assertions.\\
\newline
\begin{tabular}{rll}
$(1)$ &Suppose $f_i:[0,1]^{D-1}\longrightarrow \IR^{n_i}$ have Lipschitz constants $L_{i}$ $(1\leq i\leq q+1)$.\\
&Then $f=(f_1,...,f_{q+1}): [0,1]^{D-1}\longrightarrow \IR^{n_1+...+n_{q+1}}$ has a Lipschitz constant\\ &$\sqrt{L_{1}^2+...+L_{q+1}^2}$.\\ 
\newline
$(2)$ &Suppose $f:[0,1]^{E-1}\longrightarrow \IR^n$ has a Lipschitz constant $L$. Then for any\\ 
&$D>E$ the function $f':[0,1]^{D-1}\longrightarrow \IR^n$ defined by $f'({\bf x},{\bf x'})=f({\bf x})$  also\\ 
&has a Lipschitz constant $L$.\\
\newline
$(3)$ &Assume $f:[0,1]^E\longrightarrow \IR$, $f':[0,1]^{E'}\longrightarrow \IR^n$ are functions with Lipschitz\\ &constants $L, L'$.
          Then $\sqrt{2}\max\{\|f'\|_{\infty}L, \|f\|_{\infty}L'\}$ is a Lipschitz constant\\ 
&of the function
        $g:[0,1]^{E+E'}\longrightarrow \IR^n$ defined by $g({\bf x},{\bf x'})=f({\bf x})f'({\bf x'})$,\\ 
&where $\|f\|_{\infty}=\sup|f|, \|f'\|_{\infty}=\sup|f'|$ 
for the euclidean norms
$|f|, |f'|$.
\end{tabular}\\
\newline
Here
$(1)$ and $(2)$ are clear. To prove $(3)$
we write $f'=(f_1',...,f_n')$ so that 
\begin{alignat*}3
|g({\bf x},{\bf x'})-g({\bf y},{\bf y'})|^2=
\sum_{i=1}^{n}(f({\bf x})f_i'({\bf x'})-f({\bf y})f_i'({\bf y'}))^2
\end{alignat*}
which because of 
\begin{alignat*}3
(aa'-bb')^2=(a'(a-b)+b(a'-b'))^2\leq 2(a'^2(a-b)^2+b^2(a'-b')^2)
\end{alignat*}
is at most
\begin{alignat*}3
&2\sum_{i=1}^{n}(f_i'({\bf x'})^2(f({\bf x})-f({\bf y}))^2+
(f({\bf y})^2(f'_i({\bf x'})-f'_i({\bf y'}))^2)\\
\leq 
&2(\|f'\|_{\infty}^2L^2|{\bf x}-{\bf y}|^2+\|f\|_{\infty}^2L'^2|{\bf x'}-{\bf y'}|^2).
\end{alignat*}
Now $(3)$ follows because the squared distance between 
$({\bf x}, {\bf x'})$ and $({\bf y}, {\bf y'})$ 
is $|{\bf x}-{\bf y}|^2+|{\bf x'}-{\bf y'}|^2$.\\

Back to (\ref{paparSF}).
First we will apply $(3)$ to compute Lipschitz constants of 
the single components in (\ref{paparSF}) and then we will make use of 
$(2)$ and $(1)$ to establish the final Lipschitz constant.
According to (\ref{1caseFI}) and (\ref{2caseFI})
respectively two cases for $\Phi$ may arise. 
For the first case we have
\begin{alignat}3
\label{supno1caseFI}
\|\Phi_i^{\frac{1}{d_i}}\|_{\infty}=\|e^{\frac{\psi_i}{d_i}}u\|_{\infty}
\leq \|e^{\frac{\psi_i}{d_i}}\|_{\infty}\leq e^{\|\frac{\psi_i}{d_i}\|_{\infty}}=E_i,
\end{alignat}
say.
We may assume that the image $Im \psi$ of $\psi$ meets $\partial F$ in a point $P$ 
(for if not then we can omit $\psi$) and so by assumption
$|P|\leq r_F$. Let $P'$ be an arbitrary point in $Im \psi$. Using the Lipschitz condition 
and the triangle inequality yields $|P'|\leq r_F+\sqrt{q-1}L'$
and therefore
\begin{alignat}3
\label{supnopsi1}
\|\psi_i\|_{\infty}\leq \sqrt{q-1}L'+r_F.
\end{alignat}
If we plug this in (\ref{supno1caseFI}) we obtain 
\begin{alignat}3
\nonumber\|\Phi_i^{\frac{1}{d_i}}\|_{\infty}\leq E_i &\leq \exp\left(\frac{\sqrt{q-1}L'}{d_i}+\frac{r_F}{d_i}\right)\\
\label{estFI}
&\leq \exp\left(\frac{\sqrt{q}}{d_i}(L'+r_F)\right).
\end{alignat}
Now notice that $\|\nu\|_{\infty}= \sqrt{q}r_F$ and therefore 
$\|\exp(\nu/d_i)\|_{\infty}\leq \exp(\sqrt{q}r_F/d_i)$.
This shows that the estimate (\ref{estFI}) holds 
also in the second case (\ref{2caseFI}).\\

Next let us compute a Lipschitz constant $L_i$ of $\Phi_i^{\frac{1}{d_i}}$.
We proceed by distinguishing the cases (\ref{1caseFI})
and (\ref{2caseFI}).
For the first case we observe that
$1$ is a Lipschitz constant of $f=u$ and furthermore
$\|u\|_{\infty}=1$. Also for $f'=e^{\frac{\psi_i}{d_i}}$
we have $\|f'\|_{\infty}\leq E_i$,
and the Mean Value Theorem leads to a Lipschitz constant 
for $f'$ of the form $E_iL'/d_i$. So by $(3)$ we get a Lipschitz constant for ${\Phi_i}^{\frac{1}{d_i}}=ff'$ of the form
\begin{alignat*}3
\sqrt{2}\left(\frac{L'}{d_i}+1\right)E_i\leq \sqrt{2}(L'+1)\exp\left(\frac{\sqrt{q}}{d_i}(L'+r_F)\right)
\end{alignat*}
using (\ref{estFI}).\\
Similarly we recover the Lipschitz constant
\begin{alignat*}3
\frac{2r_F}{d_i}\exp\left(\frac{\sqrt{q}r_F}{d_i}\right)
\end{alignat*}
for $\Phi_i^{\frac{1}{d_i}}$ in the second case (\ref{2caseFI}). 
We choose
\begin{alignat}3
\label{LFI}
L_{i}=
2(L'+r_F+1)\exp\left(\frac{\sqrt{q}}{d_i}(L'+r_F)\right)
\end{alignat}
to cover both cases at once.\\

Back to (\ref{paparSF}) again. We intend to apply $(3)$
to $\Phi_i({\bf t})^{\frac{1}{d_i}}\eta_i({\bf \underline{t}^{(i)}})=ff'$.
We may assume that (\ref{bouB1}) and
the image of $\eta_{i}$ have a common point, say $Q$.
Hence by (\ref{Nineq1}) and (\ref{defcinf}) we get $|Q|\leq \sqrt{n+1}C_{\en}^{inf}$. Since $L$ is 
a Lipschitz constant of $\eta_{i}$ we see as in
(\ref{supnopsi1}) that
\begin{alignat}3
\label{supnoeta}
\|\eta_i\|_{\infty}\leq \sqrt{d_i(n+1)-1}L+\sqrt{n+1}C_{\en}^{inf}
\leq \sqrt{d_i(n+1)}(L+C_{\en}^{inf}).
\end{alignat}
Now using $(3)$ with
(\ref{estFI}), (\ref{LFI}) and (\ref{supnoeta}) yields the
Lipschitz constant 
\begin{alignat*}3
3\sqrt{d_i(n+1)}(L'+r_F+1)
\exp\left(\frac{\sqrt{q}}{d_i}(L'+r_F)\right)(L+C_{\en}^{inf})
\end{alignat*}
for the component functions in (\ref{paparSF}).
Finally we extend the component functions as in $(2)$ on
$[0,1]^{d(n+1)-1}$ to use $(1)$. 
This leads to the final Lipschitz constant
\begin{alignat*}3
3\sqrt{D}(L'+r_F+1)
\exp(\sqrt{q}(L'+r_F))(L+C_{\en}^{inf}).
\end{alignat*}
\end{proof}

\bibliographystyle{amsplain}
\bibliography{literature}

\end{document}